\documentclass{amsart}

\usepackage{amsthm, amssymb, amsmath}

\input xy
\xyoption{all}

\newcommand{\quash}[1]{}  
\newcommand{\twtimes}[1]{\stackrel{#1}{\times}}
\newcommand{\isom}{\xrightarrow{\sim}}
\newcommand{\leftexp}[2]{{\vphantom{#2}}^{#1}{#2}}

\DeclareMathOperator{\GL}{GL}
\DeclareMathOperator{\SL}{SL}
\DeclareMathOperator{\Hom}{Hom}

\DeclareMathOperator{\Aut}{Aut}
\DeclareMathOperator{\Sym}{Sym}
\DeclareMathOperator{\Spec}{Spec}

\DeclareMathOperator{\Rep}{Rep}

\DeclareMathOperator{\id}{id}

\DeclareMathOperator{\Pic}{Pic}

\DeclareMathOperator{\Lie}{Lie}
\DeclareMathOperator{\Ad}{Ad}

\DeclareMathOperator{\Frac}{Frac}

\DeclareMathOperator{\Fil}{Fil}

\def\AA{\mathbb{A}}

\def\CC{\mathbb{C}}
\def\DD{\mathbb{D}}

\def\FF{\mathbb{F}}
\def\GG{\mathbb{G}}

\def\PP{\mathbb{P}}
\def\QQ{\mathbb{Q}}
\def\RR{\mathbb{R}}

\def\ZZ{\mathbb{Z}}

\def\calE{\mathcal{E}}
\def\calF{\mathcal{F}}
\def\calO{\mathcal{O}}
\def\calP{\mathcal{P}}

\def\calI{\mathcal{I}}

\def\calK{\mathcal{K}}

\def\calL{\mathcal{L}}

\def\bR{\mathbf{R}}

\def\frg{\mathfrak{g}}
\def\frt{\mathfrak{t}}
\def\frb{\mathfrak{b}}

\def\fru{\mathfrak{u}}
\def\sl{\mathfrak{sl}}
\def\frm{\mathfrak{m}}

\def\tsigma{\widetilde{\sigma}}
\def\tiota{\widetilde{\iota}}
\def\tilh{\widetilde{h}}
\def\tcF{\widetilde{\calF}}
\def\tcK{\widetilde{\calK}}
\def\tpi{\widetilde{\pi}}

\def\Grass{\mathcal{G}r}

\def\pol{\textup{pol}}
\def\xcoch{\mathbb{X}_\bullet}
\def\xch{\mathbb{X}^\bullet}
\def\pt{\textup{pt}}
\def\Mod{\textup{Mod}}
\def\Dist{\textup{Dist}}
\def\det{\textup{det}}
\def\weight{\textup{weight}}
\def\SU{\textup{SU}}
\def\pH{\leftexp{p}{H}}
\def\Lboxtimes{\stackrel{\mathbb{L}}{\boxtimes}}
\def\Sp{\textup{Sp}}
\def\twt{\Grass_X\widetilde{\times}\Grass_X}
\def\twG{G(F)\twtimes{G(\calO)}\Grass_G}

\swapnumbers
\theoremstyle{plain}
\newtheorem{theorem}[subsection]{Theorem}
\newtheorem{lemma}[subsection]{Lemma}
\newtheorem{cor}[subsection]{Corollary}
\newtheorem{prop}[subsection]{Proposition}

\theoremstyle{definition}

\newtheorem{remark}[subsection]{Remark}

\numberwithin{equation}{section}

\title{Integral homology of loop groups via Langlands dual groups}
\author{Zhiwei Yun}
\address{Zhiwei Yun: Institute for Advanced Study, Einstein Drive, Princeton, NJ 08540, USA}
\email{zyun@math.ias.edu}
\author{Xinwen Zhu}
\address{Xinwen Zhu: Department of Mathematics, Harvard University, one Oxford Street, Cambridge, MA 02138, USA}
\email{xinwenz@math.harvard.edu}

\date{August 2009}

\begin{document}

\begin{abstract} Let $K$ be a connected compact Lie group, and $G$ be its complexification. The homology of the based loop group $\Omega K$ with integer coefficients is naturally a $\ZZ$-Hopf algebra. After possibly inverting $2$ or $3$, we identify $H_*(\Omega K,\ZZ)$ with the Hopf algebra of algebraic functions on $B^\vee_e$, where $B^\vee$ is a Borel subgroup of the Langlands dual group scheme $G^\vee$ of $G$ and $B^\vee_e$ is the centralizer in $B^\vee$ of a regular nilpotent element $e\in\Lie B^\vee$. We also give a similar interpretation for the equivariant homology of $\Omega K$ under the maximal torus action.
\end{abstract}

\maketitle

\section{Introduction}
Let $K$ be a connected compact Lie group and $G$ be its complexification. Let $\Omega K$ be the based loop space of $K$. Then the homology $H_*(\Omega K,\ZZ)$ is a Hopf algebra over $\ZZ$. The goal of this note is to describe this Hopf algebra canonically in terms of the Langlands dual group of $G$.

The rational (co)homology of $\Omega K$
is easy since rationally $K$ is the product of
spheres. However, the integral (co)homology is much more subtle: $H^*(\Omega K,\ZZ)$ is not a finitely generated $\ZZ$-algebra, and the simplest example $H^*(\Omega \SU(2),\ZZ)$ involves divided power structures. The integral (co)homology of $\Omega K$ has been studied extensively by different methods.

The first method was pioneered by Bott (cf. \cite{Bott}). He developed an algorithm,
which theoretically determines the Hopf algebra structure of
integral (co)homology of all based loop groups. By
applying this algorithm, he gave explicit descriptions of the
(co)homology of special unitary groups, orthogonal groups, and the
exceptional group $G_2$. However, Bott's description
depends on the choice of a ``generating circle'' as defined in \emph{loc. cit.} and is therefore not canonical.

The second method to study the (co)homology of based loop groups comes from the theory of Kac-Moody groups. Namely, let $G$ be the
complexification of $K$. This is a reductive algebraic group over
$\CC$. Let $F=\CC(\!(t)\!)$ and $\calO=\CC[[t]]$. Then the quotient
$\Grass_G=G(F)/G(\calO)$, known as the {\em affine Grassmannian} of $G$, is a maximal partial flag variety of the Kac-Moody group $\widehat{G(F)}$ (a central extension of $G(F)$). Let $\Omega_{\pol}K$ be the
space of polynomial maps $(S^1,1)\to (K,1_K)$. More precisely, let
us parametrize $S^1$ by $e^{i\phi}, \phi\in\RR$, and let
$K\subset SO(n,\RR)$ be an embedding. Then $\Omega_{\pol}K$ is the
space of maps from $(S^1,1)\to (K,1_K)$ such that when composed
with $K\subset SO(n,\RR)$, the matrix entries of the maps are given by
Laurent polynomials of $e^{i\phi}$. It is known that
$\Omega_{\pol}K$ is homotopic to $\Omega K$. On the other hand,
the obvious map
\[
\iota:\Omega_{\pol}K\to G(F)\to \Grass_G
\]
is a homeomorphism between $\Omega_{\pol}K$ and $\Grass_G$, which
is $K$-equivariant (the action of $K$ on $\Omega_{\pol}K$ is the pointwise
conjugation and on $\Grass_G$ is the left multiplication). See \cite{PS} for details about these facts.
Therefore, $\Grass_G$ and $\Omega K$ have the same (co)homology groups. Kostant and Kumar (cf. \cite{KK}) studied the topology of the (partial) flag varieties of arbitrary Kac-Moody groups. In particular they determined the (equivariant) cohomology rings of these partial flag varieties, using an algebraic construction called the \emph{nil-Hecke ring}. D. Peterson (unpublished work) realized that the nil-Hecke ring can also be used to study the homology ring of the affine Grassmannian. The affine variety $\Spec H_*(\Grass_G)$ is usually called the \emph{Peterson variety}.

In this note, we will proceed from yet another perspective of this story, which is pioneered by Ginzburg(cf. \cite{G}). We first recall the geometric Satake isomorphism developed by Lusztig (\cite{L}),Ginzburg (\cite{G}) and Mirkovi\'c-Vilonen (\cite{MV}). We refer the details to the paper \cite{MV}. Recall that
$\Grass_G$ is a union of projective varieties. Let $\calP_k$ be the abelian category of
$G(\calO)$-equivariant perverse sheaves with $k$-coefficients on
$\Grass_G$ whose supports are finite dimensional, where $k$ is some commutative noetherian ring of finite global dimension. The convolution
product makes $\calP_k$ into a tensor
category, equipped with the functor $H^*(\Grass_G,-)$ of taking (hyper)cohomology as a
fiber functor. The geometric Satake isomorphism claims that $H^*(\Grass_G,-)$ gives an equivalence of tensor categories
\begin{equation}\label{GS}
H^*(\Grass_G,-):\calP_k\cong\Rep(G_k^\vee,k),
\end{equation}
where $G^\vee$ is the Langlands dual group scheme of $G$, defined
over $\ZZ$. Ginzburg observed that taking cohomology also gives rise
to a tensor functor
\[H^*(\Grass_G,-):\calP_k\to\Mod(H^*(\Grass_G,k)).\]
Assume that $G$ is simple and simply-connected and $k=\CC$. By Tannakian formalism, one obtains a map of Hopf algebras
\begin{equation}\label{Ginz}
H^*(\Grass_G,\CC)\to U(\frg^\vee_{\CC}),
\end{equation}
where $\frg^\vee_{\CC}$ is the Lie algebra of $G^\vee_{\CC}$.
Ginzburg proved that this map is injective, and the first Chern
class $c_1(\calL_\det)$ of the determinant line bundle $\calL_\det$
on $\Grass_G$ gets mapped to a regular nilpotent element $e$ in
$\frg^\vee_{\CC}$ under \eqref{Ginz}. Hence the map \eqref{Ginz} factors through
$U(\frg^\vee_{\CC,e})\subset U(\frg^\vee_{\CC})$, where
$\frg^\vee_{\CC,e}$ is the centralizer of $e$ in
$\frg^\vee_{\CC,e}$. Ginzburg concluded that this gives an isomorphism
\begin{equation}\label{ge}
H^*(\Grass_G,\CC)\cong U(\frg^\vee_{\CC,e}).
\end{equation}

This note is an extension of the above result to the (co)homology of $\Grass_G$ with integer coefficients. According to \cite{MV}, the geometric Satake isomorphism holds for $k=\ZZ$. Therefore, it is natural to expect that \eqref{ge} should hold over $\ZZ$ (hence should hold modulo every prime $p$). However, this is not the case, and it fails for two reasons:

First, as we already remarked, the integral cohomology ring $H^*(\Grass_G,\ZZ)$ is not finitely generated over $\ZZ$ whereas $U(\frg^\vee_e)$ is. For this reason, we prefer to working with homology because it is a finitely generated $\ZZ$-algebra (under Pontryagin product) by a result of Bott in \cite{Bott}. The Tannakian formalism (cf. \S \ref{relHtodual}) gives us a natural homomorphism of group schemes
\begin{equation}\label{grphom}\Spec H_*(\Grass_G,\ZZ)\to G^\vee_{\ZZ}.\end{equation}

Second, although the first Chern class $c_1(\calL_\det)$ still gives rise to an element $e\in\frg^\vee_{\ZZ}$, it is not a regular nilpotent element modulo every prime $p$. When $e\mod p$ is not regular, $U(\frg^\vee_e)\otimes\FF_p$ is ``too big'' to be isomorphic to $H^*(\Grass_G,\FF_p)$. We must carefully exclude these $p$'s.

Now the main theorem of this note reads as
\begin{theorem}[See Theorem \ref{th:main}(1)]\label{th:intro} Let $G$ be a reductive connected group over $\CC$ such that its derived group $G^{der}$ is almost simple. Let $K$ be a maximal compact subgroup of $G$. Let $\ell_G$ be the square of the ratio of the lengths of long roots and the short roots of $G$ (so $\ell_G$=$1,2$ or $3$). Then there is a canonical isomorphism of group schemes over $\ZZ[1/\ell_G]$:
\[\Spec H_*(\Omega K,\ZZ[1/\ell_G])\isom B^\vee_e[1/\ell_G],\]
where $B^\vee$ is a fixed Borel subgroup of $G^\vee$, $e\in\Lie
B^\vee$ is a regular nilpotent element given by \eqref{e0}, and
$B^\vee_e$ is the centralizer of $e$ in $B^\vee$.
\end{theorem}

We will also prove an equivariant version of the above theorem. Let
$T$ be a maximal torus of $G$ such that $T\cap K$ is a maximal torus of $K$. Then $T\cap K$ acts on $\Omega K$ via
conjugation. Let $R_T=H^*_{T\cap K}(\pt,\ZZ)$. The $T\cap
K$-equivariant homology $H^{T\cap K}_*(\Omega K,\ZZ)$ (for precise
definition, see \eqref{homoT}) is an $R_T$-Hopf algebra, hence
$\Spec H^{T\cap K}_*(\Omega K,\ZZ)$ is a group scheme over $R_T$. On the
other hand, the $T$-equivariant Chern classes of line bundles on
$\Grass_G$ give an element $e^T\in\frg^\vee_{\ZZ}\otimes R_T[1/n_G]$
(where $n_G$ is an integer explicitly given in Remark \ref{nG}).
The equivariant version of the main theorem reads as

\begin{theorem}[See Theorem \ref{th:main}(2)] Notations are the same as in Theorem \ref{th:intro}. There is a canonical isomorphism of group schemes over $R_T[\dfrac{1}{\ell_Gn_G}]$:
\begin{equation*}
\Spec H^{T\cap K}_{*}(\Omega K,\ZZ[\dfrac{1}{\ell_Gn_G}])\isom B^\vee_{e^T}[1/\ell_G],
\end{equation*}
where $B^\vee_{e^T}$ is the centralizer of $e^T$ in $B^\vee\times\Spec R_T[1/n_G]$.
\end{theorem}

Finally, we also have descriptions of the cohomology $H^*(\Omega K)$ and the $K$-equivariant homology $H^K_*(\Omega K)$ in terms group-theoretic data of $G^\vee$. For details, see \S\ref{dist} and \S\ref{ss:Geq}.

\medskip

\noindent\emph{Plan of the paper.} In \S \ref{fiberfun}, we prove
that the equivariant cohomology functor $H^*_T$ is a fiber functor
from $\calP$ to $R_T:=H_T^*(\pt)$-modules that is canonically
isomorphic to $H^*\otimes R_T$. We also describe how cup product by $H^*(\Grass_G)$ interacts with the tensor structure of the functor $H^*_T$. In \S \ref{relHtodual}, we construct a canonical homomorphism
of groups schemes over $\ZZ$ from $\Spec H_*^T(\Grass_G)$ to
$B^\vee_\ZZ$. In \S \ref{line bundles}, we
digress to discuss about line bundles on the affine Grassmannian and
their equivariant structures. In \S \ref{first Chern class}, we
determine the equivariant Chern class $c_1^T(\calL_\det)$ explicitly
as an element in $\frg^\vee_\ZZ\otimes R_T$. Finally in \S
\ref{proof}, we prove our main theorem.

\medskip

\noindent\emph{Convention and Notations.} In this note, we always
assume that $G$ is a connected reductive algebraic group over $\CC$. From \S \ref{line bundles}, we will assume that $G^{der}$ is almost simple. We will fix a
Borel subgroup $B\subset G$ and a maximal torus $T\subset B$. Let
$U$ be the unipotent radical of $B$. Let $\xch(T)$ and $\xcoch(T)$
be the character group and cocharacher group of $T$. Let
$K$ be the maximal compact subgroup of $G$ containing the maximal compact torus in $T$. We shall identify $\Omega K$
with $\Omega_{\pol}K$, and denote them simply by $\Omega K$.

The Langlands dual group $G^\vee$ is the Chevalley group scheme over $\ZZ$ whose root system is identified with the coroot system of $G$. Later a Borel subgroup $B^\vee\subset G^\vee$ and a maximal torus $T^\vee\subset G^\vee$ will be fixed by the geometric Satake isomorphism. Let $U^\vee$ be the unipotent radical of $B^\vee$. The Lie algebras of these group schemes will be denoted by $\frg^\vee,\frb^\vee,\frt^\vee$ and $\fru^\vee$, which are free $\ZZ$-modules.

With dual groups in mind, our notation concerning the root system of $G$ is opposite to the usual one. Let
$\Phi^\vee,\Phi$ be the set of roots and coroots of $G$. Coweights
and coroots of $G$ will be denoted by $\lambda,\mu,\ldots$ and
$\alpha,\beta,\ldots$, while weights and roots are denoted by
$\lambda^\vee,\mu^\vee,\ldots$ and $\alpha^\vee,\beta^\vee,\ldots$.
Let $2\rho^\vee$ denote the sum of positive roots. We will give $\xcoch(T)$ a
partial order such that $\lambda\leq\mu$ if and only if
$\mu-\lambda$ is a positive integral combination of the simple
coroots. Let $(-,-)_{Kil}$ be the Killing form on $\xcoch(T)$ given by:
\[
(x,y)_{Kil}=\sum_{\alpha^\vee\in\Phi^\vee}\langle\alpha^\vee,x\rangle\langle\alpha^\vee,y\rangle.
\]

The (co)homology groups $H^*(-),H_*(-)$ are taken with $\ZZ$-coefficients unless otherwise specified. All tensor products $\otimes$ with no base ring specified are understood to be taken over $\ZZ$. For simplicity, we will write $\calP$ for $\calP_\ZZ$. For objects $\calF\in\calP$, we will abbreviate $H^*(\Grass_G,\calF),H^*_T(\Grass_G,\calF)$ by $H^*(\calF)$ and $H^*_T(\calF)$.

\medskip

\noindent\emph{Acknowledgment.} The research of Z.Y. is supported by
the National Science Foundation under the agreement No.DMS-0635607.
Any opinions, findings and conclusions or recommendations expressed
in this material are those of the authors and do not necessarily
reflect the views of the National Science Foundation. The research
of X.Z. was partially conducted during the period he was employed by
the Clay Mathematics Institute as a Liftoff Fellow.

\section{Equivariant cohomology functor}\label{fiberfun}
Let $G$ be a connected reductive algebraic group over $\CC$. Let
$F=\CC(\!(t)\!)$ and $\calO=\CC[[t]]$. Let $\Grass_G=G(F)/G(\calO)$
be the affine Grassmannian of $G$. Each coweight
$\lambda\in\xcoch(T)$ determines a point $t^\lambda\in T(F)$, and
hence a point in $\Grass_G$, which we still denote by $t^\lambda$.
For $\lambda\in\xcoch(T)$, let $\Grass_\lambda=G(\calO)t^\lambda
G(\calO)/G(\calO)$ be the $G(\calO)$-orbit through $t^\lambda$. Each
$G(\calO)$-orbit of $\Grass_G$ contains a unique point $t^\lambda$
for some {\em dominant} coweight $\lambda$. For $\lambda$ dominant,
we denote the closure of $\Grass_\lambda$ by $\Grass_{\leq\lambda}$.
Then each $\Grass_{\leq\lambda}$ is a projective variety and
$\Grass_G$ is their union.

\subsection{Weight functors and MV-filtration} In \cite{MV}, Mirkovi\'c and Vilonen introduce the weight functors
\begin{equation}\label{wt}
H^*_{c}(S_\mu,-):\calP\to\Mod(\ZZ),
\end{equation}
where $S_{\mu}$ is the $U(F)$-orbit through $t^\mu,\mu\in\xcoch(T)$
(notice that in \cite{MV}, $U$ is denoted by $N$). They show in
\cite[Theorem 3.6]{MV} that there is a natural isomorphism of {\em
tensor} functors:
\begin{equation}\label{MVsplit}
H^*(-)\cong\bigoplus_{\mu\in\xcoch(T)}H^*_c(S_\mu,-):\calP\to\Mod(\ZZ).
\end{equation}
Moreover, under the geometric Satake isomorphism \eqref{GS}, the weight functors correspond to the weight spaces of a maximal torus $T^\vee\subset G^\vee$.

For each $\calF\in\calP$, consider the filtration $\{\Fil_{\geq\mu}H^*(\calF)\}$ indexed by the partially ordered set $\xcoch(T)$:
\begin{equation*}
\Fil_{\geq\mu}H^*(\calF)=\ker(H^*(\calF)\to H^*(S_{<\mu},\calF)),
\end{equation*}
where $S_{<\mu}=\overline{S_\mu}-S_\mu$. We call this filtration the {\em MV-filtration}. It is functorial for $\calF\in\calP$. Clearly, the weight functors are the associated graded pieces of the MV-filtration. Under the geometric Satake isomorphism \eqref{GS}, there is a unique Borel subgroup $B^\vee\subset G^\vee$ containing $T^\vee$ such that the natural  $\xcoch(T)$-filtration on the $B^\vee$-module $H^*(\calF)$ coincides with the MV-filtration.

Now consider the equivariant cohomology functor:
\begin{equation}\label{eqH}
H^*_{T}(-):\calP\to\Mod^{gr}(R_T).
\end{equation}
Here $R_T=H^*_{T}(\pt,\ZZ)\isom\Sym(\xch(T))$, and $\Mod^{gr}(R_T)$ is the category of graded $R_T$-modules.

\begin{lemma}\label{l:freeT}
There is a natural isomorphism of functors
\begin{equation*}
H^*_T(\Grass,-)\isom H^*(\Grass,-)\otimes R_T:\calP\to\Mod^{gr}(R_T).
\end{equation*}
\end{lemma}

\begin{proof}
The construction of weight functors \eqref{wt} extends to the $T$-equivariant setting. For each $\mu\in\xcoch(T)$, the cohomology $H^*_c(S_\mu,\calF)$ is concentrated in degree $\langle2\rho^\vee,\mu\rangle$ (\cite[Theorem 3.5]{MV}), hence the equivariant cohomology $H^*_{T,c}(S_\mu,\calF)$ is free over $R_T$ with a {\em canonical} isomorphism
\begin{equation}\label{tensorRT}
H^*_{T,c}(S_\mu,\calF)=H^*_c(S_\mu,\calF)\otimes R_T
\end{equation}
There is a spectral sequence calculating $H^*_T(\calF)$ with $E_1$-terms $H^*_{T,c}(S_\mu,\calF)$. Over each component of $\Grass_G$, the nonzero terms of this $E_1$ page are all in degrees of the same parity, hence the spectral sequence degenerates. In other words, we get a $T$-equivariant version of the MV-filtration $\{\Fil_{\geq\mu}H^*_T(\calF)\}$ with associated graded pieces $H^*_{T,c}(S_\mu,\calF)$.

Again by \cite[Theorem 3.5]{MV}, $\Fil_{>\mu}H^*_T(\calF)$ is in
degrees $>\langle2\rho^\vee,\mu\rangle$, hence we have a canonical
splitting of the exact sequence
\begin{equation*}
0\to\Fil_{>\mu}\to\Fil_{\geq\mu}\to H^*_{T,c}(S_\mu,\calF)\to0
\end{equation*}
given by
\begin{equation*}
H^*_{T,c}(S_\mu,\calF)=H^{\langle2\rho^\vee,\mu\rangle}_c(S_\mu,\calF)\otimes
R_T\to (H^{\langle2\rho^\vee,\mu\rangle}\Fil_{\geq\mu})\otimes
R_T\to\Fil_{\geq\mu}.
\end{equation*}
where $H^k\Fil_{\geq\mu}$ means the degree $k$ part of $\Fil_{\geq\mu}$. Therefore, the $T$-equivariant MV-filtration on $H^*_{T}(\calF)$ also has a canonical splitting:
\begin{equation}\label{splitwt}
H^*_T(-)\cong\bigoplus_{\mu\in\xcoch(T)}H^*_{T,c}(S_\mu,-):\calP\to\Mod^{gr}(R_T).
\end{equation}

Combining \eqref{tensorRT},\eqref{splitwt} and \eqref{MVsplit}, we get isomorphisms of functors $\calP\to\Mod^{gr}(R_T)$:
\begin{equation*}
H^*_T(-)\cong\bigoplus_{\mu\in\xcoch(T)}H^*_{T,c}(S_\mu,-)\cong\bigoplus_{\mu\in\xcoch(T)}H^*_{c}(S_\mu,-)\otimes R_T\cong H^*(-)\otimes R_T.
\end{equation*}
\end{proof}

\subsection{Tensor structure on $H^*_T(-)$}\label{ss:tensor}
We first recall the definition of the convolution product $*$ on $\calP$, following \cite[\S 4]{MV}. Consider the diagram
\begin{equation*}
\Grass_G\times\Grass_G\xleftarrow{p}G(F)\times\Grass_G\xrightarrow{q}\twG\xrightarrow{m_G}\Grass_G
\end{equation*}
Here $p$ is the natural projection morphism, $q$ is the quotient map by the right $G(\calO)$-action on $G(F)\times\Grass_G$ given by $(x,y)\cdot g=(xg,g^{-1}y)$ and $m_G$ is induced by the multiplication on $G(F)$. By definition,
\begin{equation*}
\calF_1*\calF_2=m_{G,*}\tcF,
\end{equation*}
where $\tcF$ is the unique perverse sheaf on $\twG$ such that
\begin{equation*}
q^*(\tcF)=p^*\pH^0(\calF_1\Lboxtimes\calF_2).
\end{equation*}
Hence
\begin{equation}\label{defconv}
H^*_T(\calF_1*\calF_2)\cong H^*_T(\twG,\tcF).
\end{equation}

Next we construct the tensor structure on the functor $H^*_T(-)$, following the same lines as in \cite[\S 5,\S 6]{MV}. Let $X=\AA^1_{\CC}$. Recall that the global counterpart $\Grass_X$ of $\Grass_G$ classifies triples $(x,\calE,\tau)$ where $x\in X$, $\calE$ is a $G$-torsor over $X$ and $\tau$ is a trivialization of $\calE$ over $X-\{x\}$. The global counterpart of $G(F)\twtimes{G(\calO)}\Grass_G$ is the space $\twt$ which classifies tuples $(x_1,x_2,\calE_1,\calE_2,\tau_1,\tau_2)$ where $x_1,x_2\in X$, $\calE_1,\calE_2$ are $G$-torsors over $X$, $\tau_1$ is a trivialization of $\calE_1$ on $X-\{x_1\}$ and $\tau_2$ is an isomorphism of $G$-torsors $\calE_1|_{X-\{x_2\}}\isom\calE_2|_{X-\{x_2\}}$. For $x_1\neq x_2$, the fiber $(\twt)_{x_1,x_2}$ is isomorphic to $\Grass_G\times\Grass_G$; while for $x_1=x_2$, the fiber $(\twt)_{x_1,x_2}$ is isomorphic to $\twG$.

Given objects $\calF_1,\calF_2\in\calP$, we can ``spread them over the
curve $X$'' to get perverse sheaves $\calK_1,\calK_2$ on $\Grass_X$;
we can also mimic the construction of $\tcF$ to get a perverse sheaf
$\tcK$ on $\twt$ (see \cite[\S 5]{MV} for details). The restriction
of $\tcK$ on $(\twt)_{x_1,x_2}$ can be identified with
$\pH^0(\calF_1\Lboxtimes\calF_2)$ when $x_1\neq x_2$ and with $\tcF$
when $x_1=x_2$, under the above identification of fibers. In
\cite[Lemma 6.1, eqn.(6.4)]{MV}, it is shown that the direct image
sheaves $\bR^i\tpi_*(\tcK)$ are constant (as $T$-equivariant
sheaves, as we can easily see), where $\tpi:\twt\to X^2$ is the
natural projection. In other words, we have a co-specialization
isomorphism from the stalks of $\bR^i\tpi_*\tcK$ along the diagonal
$\Delta(X)\subset X^2$ to its stalks over $X^2-\Delta(X)$. Using the
above identifications, the co-specialization map takes the form:
\begin{equation}\label{sp}
\Sp^*_{\tcK}:H^*_T(\twG,\tcF)\isom H^*_T(\Grass_G\times\Grass_G,\pH^0(\calF_1\Lboxtimes\calF_2)).
\end{equation}
The argument of \cite[Lemma 6.1, eqn.(6.3)]{MV} has an obvious
$T$-equivariant version (using Lemma \ref{l:freeT}), hence
\begin{equation}\label{exttensor}
H^*_T(\Grass_G\times\Grass_G,\pH^0(\calF_1\Lboxtimes\calF_2))\cong H^*_T(\calF_1)\otimes_{R_T}H^*_T(\calF_2).
\end{equation}
Combining \eqref{sp},\eqref{defconv} and \eqref{exttensor}, we get the desired tensor structure
\begin{equation*}
H^*_T(\calF_1*\calF_2)\isom H^*_T(\calF_1)\otimes_{R_T}H^*_T(\calF_2).
\end{equation*}

\begin{lemma}
Under the tensor structure of $H^*_T(-)$ introduced above and the tensor structure of $H^*(-)\otimes R_T$ induced from that of $H^*(-)$, the natural isomorphism in Lemma \ref{l:freeT} is a tensor isomorphism.
\end{lemma}
\begin{proof}
Using the same argument of \cite[\S 6]{MV}, one can show that the sum of the $T$-equivariant weight functors has a natural tensor structure. Moreover, the isomorphism in \eqref{splitwt} is a tensor isomorphism. Using \eqref{tensorRT} and the argument of \cite[Proposition 6.4]{MV}, it is easy to see that
\begin{equation}\label{sumwt}
\bigoplus_{\mu\in\xcoch(T)}H^*_{T,c}(S_\mu,-)\cong\bigoplus_{\mu\in\xcoch(T)}H^*_{c}(S_\mu,-)\otimes R_T:\calP\to\Mod^{\xcoch(T)}(R_T)
\end{equation}
as tensor functors. The natural isomorphism $H^*_T(-)\cong H^*(-)\otimes R_T$ in Lemma \ref{l:freeT} is given by the composition of tensor isomorphisms \eqref{splitwt},\eqref{sumwt} and \eqref{MVsplit}, hence it is also a tensor isomorphism.
\end{proof}

\begin{remark}\label{r:torsorE}
Similarly we can consider
\begin{equation*}
H^*_{G(\calO)}(-)=H^*_G(-):\calP\to\Mod^{gr}(R_G).
\end{equation*}
Here $R_G=H^*_{G}(\pt,\ZZ)$. Using the argument of \cite[Proposition
6.1]{MV}, one can show that $H^*_G(-)$ is also a tensor functor.
Further argument shows that it is actually a fiber functor (i.e., it
is exact and faithful). By \cite{Sa}, this fiber functor defines a
$G^\vee$-torsor $\calE$ over $\Spec R_G$ (here we use the fact that
$G^\vee$ is flat over $\ZZ$). Moreover, by Lemma \ref{l:freeT}, the
pull back of $\calE$ to $\Spec R_T$ admits a canonical
trivialization. However, $\calE$ itself does not have a canonical
trivialization.
\end{remark}

\subsection{Equivariant (co)homology of $\Grass_G$}\label{ss:eqhomo} Since $H^*_T(\Grass_G)=H^*_{T\cap K}(\Omega K)$ and $\Omega K$ is a homotopy commutative $H$-space, $H^*_T(\Grass_G)$ is naturally a commutative and co-commutative Hopf algebra over $R_T$. By the cell-decomposition of $\Grass_G$, $H^*_T(\Grass_G)$ is a free $R_T$-module concentrated in even degrees. We denote the coproduct on $H^*_T(\Grass_G)$ by $\Delta$. 

We define the $T$-equivariant homology
\begin{equation}\label{homoT}
H^T_*(\Grass_G):=\Hom_{R_T}(H^*_T(\Grass_G),R_T)^{gr}
\end{equation}
to be the graded dual of $H^*_T(\Grass_G)$ as an $R_T$-module, i.e., $H^T_n(\Grass_G)$ sends $H^i_{T}(\Grass_G)$ to $R_T^{i-n}$. Note that $H^T_*(\Grass_G)\cong\varinjlim H^T_*(\Grass_{\leq\lambda})$ where $H^T_*(\Grass_{\leq\lambda})=H^{-*}_T(\Grass_{\leq\lambda},\DD_{\leq\lambda})$ ($\DD_{\leq\lambda}$ is the dualizing complex of $\Grass_{\leq\lambda}$ with the canonical $T$-equivariant structure).

As the $R_T$-dual of $H_T^*(\Grass_G)$, the $T$-equivariant homology $H^T_*(\Grass_G)$ is also a commutative and co-commutative Hopf algebra over $R_T$. The multiplication on $H^T_*(\Grass_G)$ is identified with the Pontryagin product on $H^{T\cap K}_*(\Omega K)$, and is denoted by $\wedge$.

Similarly, we can define the $G$-equivariant homology $H^G_*(\Grass_G)$ of $\Grass_G$ either as the $R_G$-dual of $H^*_G(\Grass_G)$ or the direct limit of $H^{-*}_G(\Grass_{\leq\lambda},\DD_{\leq\lambda})$. This is a Hopf algebra over $R_G$.

Next, we consider the action of $H^*_T(\Grass_G)$ on the fiber functor $H^*_T(\Grass_G,-)$ via cup product.

\begin{prop}\label{p:multact}
For any $\calF_1,\calF_2\in\calP$, and any $h\in H^*_T(\Grass_G)$, we have a commutative diagram
\begin{equation*}
\xymatrix{H^*_T(\calF_1)\otimes_{R_T}H^*_T(\calF_2)\ar[r]^(.6){\sim}\ar[d]^{\cup\Delta(h)} & H^*_T(\calF_1*\calF_2)\ar[d]^{\cup h}\\
H^*_T(\calF_1)\otimes_{R_T}H^*_T(\calF_2)\ar[r]^(.6){\sim} & H^*_T(\calF_1*\calF_2)}
\end{equation*}
where the horizontal maps are given by the tensor structure of the functor $H^*_T(-)$ (see the proof of Lemma \ref{l:freeT}).

The similar statement also holds for $G$-equivariant cohomology $H_G^*$.
\end{prop}
\begin{proof}
From \eqref{defconv} we see that the action of $h\in H^*_T(\Grass_G)$ on $H^*_T(\calF_1*\calF_2)$ is the same as the action of $m_G^*(h)\in H^*_T(\twG)$ on $H^*_T(G(F)\twtimes{G(\calO)}\Grass_G,\tcF)$.

Let us keep track of the cup product by the cohomology of the
relevant spaces in the construction of the tensor structure in \S\ref{ss:tensor}. Since the morphism
$\tpi:\twt\to X^2$ is stratified by affine space bundles (see
\cite[Lemma 6.1, argument for (6.4)]{MV}), an easy spectral sequence argument shows that the sheaves
$\bR^i\tpi_*\ZZ$ are locally constant, hence constant. In other
words, we also have a co-specialization isomorphism from the stalks of $\bR^i\tpi_*\ZZ$
along $\Delta(X)$ to stalks elsewhere, i.e.,
\begin{equation}\label{sph}
\Sp^*:H^*_T(\twG)\isom H^*_T(\Grass_G\times\Grass_G).
\end{equation}
Moreover, by the naturality of co-specialization maps, \eqref{sp} and \eqref{sph} are compatible under the cup product:
\begin{equation*}
\Sp^*_{\tcK}(\tilh\cup v)=\Sp^*(\tilh)\cup\Sp^*_{\tcK}(v),
\end{equation*}
for any $\tilh\in H^*_T(\twG),v\in H^*_T(\twG,\tcF)$.

From the above discussion, the cup product action of $h\in H^*_T(\Grass_G)$ on $H^*_T(\calF_1*\calF_2)$, when transported to $H^*_T(\calF_1)\otimes_{R_T} H^*_T(\calF_2)$ under the isomorphism given by the tensor structure of $H^*_T(-)$, is the cup product action of the class $\Sp^*(m_G^*(h))\in H^*_T(\Grass_G\times\Grass_G)=H^*_T(\Grass_G)\otimes_{R_T}H^*_T(\Grass_G)$. Therefore, to prove the proposition, it suffices to show that
\begin{equation}\label{toprove}
\Delta=\Sp^*m^*_G:H^*_T(\Grass_G)\to H^*_T(\Grass_G\times\Grass_G).
\end{equation}
Recall that the coproduct $\Delta$ is induced from the multiplication $m_K:\Omega K\times\Omega K\to\Omega K$, hence
\begin{equation}\label{cop}
(\iota\times\iota)^*\Delta=m_K^*\iota^*:H^*_T(\Grass_G)\to H^*_{T\cap K}(\Omega K\times\Omega K).
\end{equation}

Now we construct a map $\tiota:\Omega K\times\Omega K\times
X^2\to\twt$. For a point $(\gamma_1,\gamma_2:S^1\to
K,x_1,x_2)\in\Omega K\times\Omega K\times X^2$, we consider the
complexifications $\gamma^{\CC}_1,\gamma^{\CC}_2:\AA^1-\{0\}\to G$
(recall that $\gamma_i$ are polynomial maps so that the
complexifications make sense). Let
$\gamma^{\CC}_{i,x_i}:\AA^1-\{x_i\}\to G$ be the composition of the
translation $\AA^1-\{x_i\}\isom\AA^1-\{0\}$ with $\gamma^{\CC}_i$.
We define
\begin{equation*}
\tiota(\gamma_1,\gamma_2,x_1,x_2)=(x_1,x_2,\calE_1=\calE^{triv},\calE_2=\calE^{triv},\tau_1,\tau_2)
\end{equation*}
where $\calE^{triv}$ is the trivial $G$-torsor on $X$, and the isomorphism $\tau_i$ (viewed as an automorphism of $\calE^{triv}$ over $X-\{x_i\}$) is given by $\gamma^{\CC}_{i,x_i}:X-\{x_i\}\to G$ for $i=1,2$.

The restriction of $\tiota$ to points $x_1=x_2$ can be identified with the map $\iota_{(2)}:\Omega K\times\Omega K\xrightarrow{\iota\times\iota}G(F)\times\Grass_G\xrightarrow{q}\twG$; the restriction of $\tiota$ to $x_1\neq x_2$ is exactly the map $\iota\times\iota:\Omega K\times\Omega K\to\Grass_G\times\Grass_G$. We view $\Omega K\times\Omega K\times X^2$ as a constant family over $X^2$, then $\tiota$ commutes with the co-specialization maps, i.e., we have a commutative diagram
\begin{equation*}
\xymatrix{H^*_T(\twG)\ar[r]^{\Sp^*}\ar[d]^{\iota^*_{(2)}} & H^*_T(\Grass_G\times\Grass_G)\ar[d]^{(\iota\times\iota)^*}\\
H^*_{T\cap K}(\Omega K\times\Omega K)\ar@{=}[r] & H^*_{T\cap
K}(\Omega K\times\Omega K).}
\end{equation*}

On the other hand, we also have a commutative diagram
\begin{equation*}
\xymatrix{\Omega K\times\Omega K\ar[r]^{m_K}\ar[d]^{\iota_{(2)}} & \Omega K\ar[d]^{\iota}\\
\twG\ar[r]^(.6){m_G} & \Grass_G}
\end{equation*}
Combining the two commutative diagrams, we get
\begin{equation*}
(\iota\times\iota)^*\Sp^*m_G^*=\iota_{(2)}^*m_G^*=m_K^*\iota^*.
\end{equation*}
In view of \eqref{cop}, we get
\begin{equation*}
(\iota\times\iota)^*\Delta=(\iota\times\iota)^*\Sp^*m_G^*.
\end{equation*}
Since $(\iota\times\iota)^*$ is an isomorphism, the equality \eqref{toprove} follows. This finishes the proof of the proposition.
\end{proof}


\section{Relating $H^T_*(\Grass_G)$ to the Langlands dual group}\label{relHtodual}

Consider the fiber functor
\[
H^*_T(-)\otimes_{R_T} H^T_*(\Grass_G):\calP\to\Mod^{gr}(H^T_*(\Grass_G)).
\]
We now construct a canonical tensor automorphism $\sigma_{can}$ of this functor. For any $\calF\in\calP$, the action of $\sigma_{can}$ on $H^*_T(\calF)\otimes_{R_T}H^T_*(\Grass_G)$ is given by
\[
\sigma_{can}(v\otimes h)=\sum_{i}(h^i\cup v)\otimes(h_i\wedge h),
\]
where $v\in H^*_T(\calF), h\in H^T_*(\Grass_G)$ and $\{h^i\},\{h_i\}$ are dual bases of the free $R_T$-modules $H^*_T(\Grass_G)$ and $H^T_*(\Grass_G)$. One readily checks that $\sigma_{can}$ does not depend on the choice of the dual bases $\{h^i\},\{h_i\}$.

\begin{lemma}
The natural transformation $\sigma_{can}$ is a tensor automorphism of the fiber functor $H^*_T(-)\otimes_{R_T} H^T_*(\Grass_G)$.
\end{lemma}
\begin{proof}
Let $\calF_1,\calF_2\in\calP$ and $v_i\in H^*_T(\calF_i)$ for $i=1,2$. We have to show that
\begin{equation*}
\sigma_{can}(v_1\otimes v_2)=\sigma_{can}(v_1)\otimes\sigma_{can}(v_2).
\end{equation*}
Here we view $v_1\otimes v_2$ as an element in $H^*_T(\calF_1*\calF_2)$ under the identification $H^*_T(\calF_1)\otimes_{R_T} H^*_T(\calF_2)\cong H^*_T(\calF_1*\calF_2)$.

On one hand, we have
\begin{eqnarray}\label{tensig}
\sigma_{can}(v_1)\otimes\sigma_{can}(v_2)&=&\sum_{i,j}(h^i\cup v_1\otimes h_i)\otimes(h^j\cup v_2)\otimes h_j\\
\notag &=& \sum_{i,j}\left((h^i\cup v_1)\otimes (h^j\cup
v_2)\right)\otimes(h_i\wedge h_j).
\end{eqnarray}
On the other hand, by Proposition \ref{p:multact},
\begin{equation}\label{sigten}
\sigma_{can}(v_1\otimes v_2)=\sum_{k}\left(\Delta(h^k)(v_1\otimes v_2)\right)\otimes h_k.
\end{equation}
Since the coproduct $\Delta$ on $H^*_T(\Grass_G)$ is adjoint to the Pontryagin product $\wedge$ on $H^T_*(\Grass_G)$, the two expressions \eqref{tensig} and \eqref{sigten} are the same.
\end{proof}

By this lemma and the Tannakian formalism, $\sigma_{can}$ defines an $H^T_*(\Grass_G)$-valued point of the group scheme $\Aut^{\otimes}(H^*_T)$, which is canonically isomorphic to $G^\vee\times\Spec R_T$ by Lemma \ref{l:freeT}:
\begin{equation}\label{definetau}
\tsigma^T:\Spec H^T_*(\Grass_G)\to G^\vee\times\Spec R_T.
\end{equation}
We also have the non-equivariant counterpart:
\begin{equation}\label{definetau0}
\tsigma:\Spec H_*(\Grass_G)\to G^\vee.
\end{equation}

\begin{lemma} The morphisms $\tsigma^T$ and $\tsigma$ as in (\ref{definetau}) and (\ref{definetau0}) are homomorphisms of group schemes over $R_T$ and $\ZZ$.
\end{lemma}
\begin{proof}
We give the argument for the non-equivariant version, and the equivariant version is similar. We abbreviate $H_*(\Grass_G)$ by $H_*$ and $H^*(\Grass_G)$ by $H^*$. We need to check that the following diagram is commutative
\[
\xymatrix{\calO(G^\vee)\ar[rr]^{\tilde{\tau}}\ar[d]^{\mu^*} && H_*\ar[d]^{\Delta_*}\\
\calO(G^\vee)\otimes\calO(G^\vee)\ar[rr]^{\tilde{\tau}\otimes\tilde{\tau}} && H_*\otimes H_*}
\]
where $\mu^*$ and $\Delta_*$ are the coproducts on $\calO(G^\vee)$ and $H_*$. The composition $(\tilde{\tau}\otimes\tilde{\tau})\circ\mu^*$ corresponds to the following automorphism of the tensor functor $H^*(-)\otimes H_*\otimes H_*$:
\begin{eqnarray}\label{ex1}
\notag H^*(\calF)\otimes H_*\otimes H_*&\to& H^*(\calF)\otimes H_*\otimes H_*\\
v\otimes1\otimes1&\mapsto&\sum_{i,j}(h^i\cup h^j\cup v)\otimes h_i\otimes h_j.
\end{eqnarray}
for $v\in H^*(\calF)$ and $\{h^i\},\{h_i\}$ dual bases of $H^*$ and $H_*$ as before. On the other hand, the composition $\Delta_*\circ\tilde{\tau}$ is given by the automorphism
\begin{eqnarray}\label{ex2}
\notag H^*(\calF)\otimes H_*\otimes H_*&\to& H^*(\calF)\otimes H_*\otimes H_*\\
v\otimes1\otimes1&\mapsto&\sum_{i}(h^k\cup v)\otimes\Delta_*(h_k).
\end{eqnarray}
Since $\Delta_*:H_*\to H_*\otimes H_*$ is adjoint to $\cup:H^*\otimes H^*\to H^*$, the two expressions \eqref{ex1} and \eqref{ex2} are the same.
\end{proof}

\begin{prop}\label{p:ZU}
\begin{enumerate}
\item []
\item The homomorphism $\tsigma^T$ factors through $B^\vee\times\Spec R_T$, and is $\GG_m$-equivariant with respect to the natural grading on $R_T$ and the $\GG_m$-action on $G^\vee$ given by the conjugation by $2\rho^\vee:\GG_m\to T^\vee$;
\item We have a natural isomorphism of group schemes $\Spec H_0(\Grass_G)\isom Z^\vee$, the center of $G^\vee$;
\item If $G$ is simply-connected, $\tsigma$ factors through $U^\vee$;
\item In general, let $G^{sc}$ be the simply-connected cover of $G^{der}$, and we identify the neutral component $\Grass_G^0$ with $\Grass_{G^{sc}}$. Then we have a natural isomorphism of Hopf algebras
\[
H_*(\Grass_G)\cong H_0(\Grass_G)\otimes H_*(\Grass_{G^{sc}}).
\]
The morphism $\tsigma:\Spec H_*(\Grass_G)\to B^\vee$ factors as
\[
\Spec H_*(\Grass_G)\to \Spec H_0(\Grass_G)\times\Spec H_*(\Grass_{G^{sc}})\xrightarrow{*} Z^\vee\times U^\vee\hookrightarrow B^\vee.
\]
where the starred arrow is the product of the morphisms in (2) and (3). Here we have identified unipotent radicals $U^\vee$ of $B^\vee$ for $G^\vee$ and $(G^{sc})^\vee=(G^\vee)^{ad}$.
\end{enumerate}
\end{prop}
\begin{proof}
(1) The action of $\sigma_{can}$ preserves the MV-filtration on $H^*_T(-)\otimes_{R_T}H^T_*(\Grass_G)$, which implies that $\tsigma^T$ factors through $B^\vee\times\Spec R_T$. On the other hand, the action of $\sigma_{can}$ also preserves the grading on $H^*_T(-)\otimes_{R_T}H^T_*(\Grass_G)$, which implies that $\tsigma^T$ is $\GG_m$-equivariant.

(2) The component group $\pi_0(\Grass_G)\cong\pi_1(K)\cong\xcoch(T)/\ZZ\Phi$ (here $\ZZ\Phi$ is the coroot lattice of $G$). Therefore $H_0(\Grass_G)$ is the group algebra of $\xcoch(T)/\ZZ\Phi$, hence isomorphic to $\calO(Z^\vee)$ as Hopf algebras.

(3) and (4) Since the tensor automorphism $\sigma_{can}$ preserves the MV-filtration, it induces a tensor automorphism $\bar{\sigma}_{can}$ on the associated graded pieces, i.e., on the tensor functor
\[
\bigoplus_{\mu\in\xcoch(T)} H^*_c(S_\mu,-)\otimes H_*(\Grass_G):\calP\to\Mod^{\xcoch(T)}(H_*(\Grass_G))
\]
The tensor automorphism $\bar{\sigma}_{can}$ gives a homomorphism $\Spec H_*(\Grass_G)\to T^\vee$ which is clearly the same as the composition $\Spec H_*(\Grass_G)\stackrel{\tsigma}{\to} B^\vee\stackrel{\pi}{\to} T^\vee$. On the other hand, since $H_c^*(S_\mu,\calF)$ is concentrated in one degree, the action of $\bar{\sigma}_{can}$ on it is given by
\[
v\mapsto (\bar{\mu}^*\cup v)\otimes\bar{\mu}\in H^*_c(S_\mu,\calF)\otimes H_0(\Grass_G).
\]
Here $v\in H^*_c(S_\mu,\calF)$, $\bar{\mu}$ is the image of $\mu\in\xcoch(T)$ in $\xcoch(T)/\ZZ\Phi^\vee\subset\ZZ[\xcoch(T)/\ZZ\Phi^\vee]= H_0(\Grass_G)$, and $\bar{\mu}^*$ is the element in $H^0(\Grass_G)$ that takes value $1$ on the component corresponding to $\bar{\mu}$ and $0$ on other components. This means that the tensor automorphism $\bar{\sigma}_{can}$ in fact comes from a tensor automorphism of the functor
\[
\bigoplus_{\mu\in\xcoch(T)} H^*_c(S_\mu,-)\otimes H_0(\Grass_G):\calP\to\Mod^{\xcoch(T)}(H_0(\Grass_G))
\]
followed by the functor $\Mod^{\xcoch(T)}(H_0(\Grass_G))\to\Mod^{\xcoch(T)}(H_*(\Grass_G))$ given by $\otimes_{H_0}H_*(\Grass_G)$. In other words, there is a commutative diagram
\begin{equation}\label{proZ}
\xymatrix{\Spec H_*(\Grass_G)\ar[d]\ar[rr]^{\tsigma} & & B^\vee\ar[d]^{\pi}\\
\Spec H_0(\Grass_G)\ar[r]^(0.7){\sim} & Z^\vee\ar@{^(->}[r] & T^\vee}
\end{equation}
which implies that $\tsigma$ factors through $\pi^{-1}(Z^\vee)=Z^\vee\times U^\vee\subset B^\vee$.

Now if $G=G^{sc}$, $Z^\vee$ is trivial and the map $H_*(\Grass_G)\to B^\vee$ factors through $U^\vee$. This proves (3). In general, by the functoriality of the geometric Satake isomorphism for the homomorphism $G^{sc}\to G$, we have a commutative diagram
\begin{equation}\label{proU}
\xymatrix{\Spec H_*(\Grass_G)\ar[r]\ar[d] & Z^\vee\times U^\vee\ar[d]\\
\Spec H_*(\Grass_{G^{sc}})\ar[r] & U^\vee}
\end{equation}
Combining \eqref{proZ} and \eqref{proU}, the assertion is proved.
\end{proof}

\begin{remark}\label{r:loc}
The composition
\begin{equation}\label{prT}
\Spec H^T_*(\Grass_G)\xrightarrow{\tsigma^T} B^\vee\times\Spec R_T\to T^\vee\times\Spec R_T
\end{equation}
can be interpreted via equivariant localization. In fact, as we remarked in the proof of Proposition \ref{p:ZU}(3)(4), the above morphism $\Spec H^T_*(\Grass_G)\to T^\vee\times\Spec R_T$ comes from the tensor automorphism $\bar{\sigma}_{can}$ of the $T$-equivariant weight functors \eqref{sumwt}. On the other hand, the cup product action of $H^*_T(\Grass_G)$ on $H^*_{T,c}(S_\lambda,-)$ factors through the restriction map $i_\lambda^*:H^*_T(\Grass_G)\to H^*_T(S_\lambda)\isom H^*_T(\{t^\lambda\})$ (the second arrow is an isomorphism because $S_\lambda$ contracts to $t^\lambda$ under a $\GG_m$-action). Therefore, for $v_\lambda\in H^*_{T,c}(S_\lambda,\calF)$, we have
\begin{equation*}
\bar{\sigma}_{can}(v_\lambda)=\sum_i(i_\lambda^*(h^i)\cup v_\lambda)\otimes h_i=v_\lambda\otimes i_{\lambda,*}(1),
\end{equation*}
where $i_{\lambda,*}:R_T=H^T_*(\{t^\lambda\})\to H^T_*(\Grass_G)$ is the adjoint of $i^*_\lambda$. The product of $i^*_{\lambda}$ is the equivariant localization map:
\begin{equation*}
Loc^*=\prod_{\lambda\in\xcoch(T)} i^*_{\lambda}:H^*_T(\Grass_G)\to\prod H^*_T(\{t^\lambda\}).
\end{equation*}
Let
\begin{equation}\label{homloc}
Loc_*=\sum_{\lambda\in\xcoch(T)} i_{\lambda,*}:R_T[\xcoch(T)]\to H^T_*(\Grass_G)
\end{equation}
be the adjoint of $Loc^*$, where $R_T[\xcoch(T)]$ is the $R_T$-valued group ring of the abelian group $\xcoch(T)$. The above discussion shows that the morphism \eqref{prT} coincides with
\begin{equation*}
\Spec(Loc_*):\Spec H^T_*(\Grass_G)\to\Spec R_T[\xcoch(T)]=T^\vee\times\Spec R_T.
\end{equation*}
\end{remark}


\section{Remarks on line bundles}\label{line bundles}
From this section on, we assume that $G^{der}$ is almost simple. In this case, it is well-known that the Picard group of each component of $\Grass_G$ is isomorphic to $\ZZ$. Let $\calL_{\det}$ denote the line bundle on $\Grass_G$ whose restriction to each component is the positive generator of the Picard group.

\subsection{Equivariant line bundles} Fix a representation $V$ of $G$. The homomorphism $\phi:G\to\GL(V)$ induces a morphism of ind-schemes $\Grass_{\phi}:\Grass_G\to\Grass_{\GL(V)}$ over $\CC$. We identify $\Grass_{\GL(V)}$ with the set of $\calO$-lattices in $V\otimes F$ in the usual way. Fix a lattice $\Lambda_1\in\Grass_{\GL(V)}$. Define a line bundle $\calL_{\Lambda_1}$ whose value at a point $\Lambda\in\Grass_{\GL(V)}$ is the line
\begin{equation*}
\det(\Lambda:\Lambda_1):=\det(\Lambda/\Lambda\cap\Lambda_1)\otimes\det(\Lambda_1/\Lambda\cap\Lambda_1)^{\otimes-1}
\end{equation*}
where $\det(-)$ means taking the top exterior power of a finite dimensional $\CC$-vector space. Let $\calL^\phi_{\Lambda_1}=\Grass_{\phi}^*\calL_{\Lambda_1}$ be the pull-back line bundle on $\Grass_G$.

Let $G(F)_{\Lambda_1}$ be the stabilizer of the lattice $\Lambda_1$ in $G(F)$(via the action $\phi$). Then $\calL^\phi_{\Lambda_1}$ has a natural $G(F)_{\Lambda_1}$-equivariant structure: for each $g\in G(F)_{\Lambda_1}$, we identify $\det(\Grass_\phi(x):\Lambda_1)$ and $\det(\Grass_\phi(gx):\Lambda_1)=\det(\Grass_\phi(gx):g\Lambda_1)$ via the action of $g$. Let $H\subset G(F)_{\Lambda_1}$ be an algebraic subgroup. Let $x\in\Grass_G$ and $H_x$ be the stabilizer of $x$ in $H$, so that we can identify the $H$-orbit through $x$ with the homogeneous space $H/H_x$. Via the induction functor, we have an identification
\begin{equation*}
\Pic^{H}(H\cdot x)=\Pic^H(H/H_x)=\xch(H_x).
\end{equation*}
It is easy to see that the image of $\calL^\phi_{\Lambda_1}|_{H\cdot x}\in\Pic^{H}(H\cdot x)$ in $\xch(H_x)$ is the character given by the action of $H_x$ on the line $\det(\Grass_{\phi}(x):\Lambda_1)$ (notice that $H_x$ stabilizes both $\Grass_{\phi}(x)$ and $\Lambda_1$).

We consider the special case where $\Lambda_0=V\otimes_{\CC}\calO$. Since $\Lambda_0$ is stabilized by $G(\calO)$, $\calL^\phi_{\Lambda_0}$ has a natural $G(\calO)$-equivariant structure. Take $H=G\subset G(\calO)$ and $x=t^\lambda$. Then $H_x=P_\lambda\subset G$ is a parabolic subgroup containing the maximal torus $T$. We can identify the image of $\calL^\phi_{\Lambda_0}$ in $\xch(P_\lambda)\subset\xch(T)$ as the character of the action of $T$ on $\det(t^\lambda\Lambda_0:\Lambda_0)$, which is
\begin{equation}\label{reschern}
-\sum_{\chi^\vee\in\weight(V)}\dim V_{\chi^\vee}\cdot\langle\chi^\vee,\lambda\rangle\chi^\vee.
\end{equation}
We can further specialize to the case $H=H_x=T$, then the element in \eqref{reschern} is also the restriction of $c^T_1(\calL^\phi_{\Lambda_0})$ to the point $t^{\lambda}$, if we identify $\xch(T)$ with $H^2_T(\pt)$, .

\begin{lemma}\label{l:deg}(See also \cite[Lemma 10.6.1]{Sor}) We have $\calL^\phi_{\Lambda_1}=\calL^{\otimes d_V}_{\det}\in\Pic(\Grass_G)$ where
\[
d_V=\frac{1}{2}\sum_{\chi^\vee\in\weight(V)}\dim V_{\chi^\vee}\cdot\langle\chi^\vee,\theta\rangle^2.
\]
Here $\theta$ stands for the coroot corresponding to the highest root $\theta^\vee$ of $G$.
\end{lemma}
\begin{proof}For $i\in\pi_0(\Grass_G)$, let $\Grass^i_G$ denote the corresponding connected component. We first claim that the degree of $\calL^\phi_{\Lambda_1}$ on the generating cycle of $H_2(\Grass^i_G)(\cong\ZZ)$ is independent of the component $\Grass_G^i$. In fact, for any $g\in G(F)$, by choosing a trivialization of the line $\det(g\Lambda_1:\Lambda_1)$, we get an isomorphism of line bundles $g^*\calL^\phi_{\Lambda_1}\cong\calL^\phi_{\Lambda_1}$. Since $G(F)$ transitively permutes the components of $\Grass_G$, our assertion follows. In the following we concentrate on the neutral component $\Grass^0_G$.

Let $I\subset G(\calO)$ be the Iwahori subgroup of $G(F)$ containing $B$. The only 1-dimensional $I$-orbit in $\Grass^0_G$ is $It^{-\theta}$, and the fundamental class of its closure generates $H_2(\Grass^0_G)$. Therefore the number $d_V$ in question is the degree of $\calL^\phi_{\Lambda_1}$ on the closure of $It^{-\theta}$. 

Fix an $\sl(2)$-triple $(x_{\theta^\vee},\theta,x_{-\theta^\vee})$ where $x_{\pm\theta^\vee}$ belongs to the root space $\frg_{\theta^\vee}$. Let $\phi_0:\SL(2)\to G(F)$ be the homomorphism which sends the standard triple $(e,h,f)$ to $(t^{-1}x_{\theta^\vee},\theta,tx_{-\theta^\vee})$. Then the closure of $It^{-\theta}$ in $\Grass_G$ is the $\phi_0(\SL(2))$-orbit through $t^{-\theta}$, and the stabilizer of $t^{-\theta}$ is the upper triangular matrices $B_2\subset\SL(2)$. To compute the degree of $\calL^\phi_{\Lambda_1}$ on this orbit, we will choose a lattice $\Lambda_1$ which is stable under the $\SL(2)$-action and calculate the character of $B_2$ acting on $\det(t^{-\theta}\Lambda_0:\Lambda_1)$.

Let $V=\oplus_nV_n$ be the grading according to the action of $\theta$: $V_n$ is the direct sum of weight spaces $V_{\chi^\vee}$ such that $\langle\chi^\vee,\theta\rangle=n$. Then the lattice $\Lambda_1=\oplus_nt^{-[n/2]}\calO\otimes_{\CC}V_n$ is stable under the $\SL(2)$-action. On the other hand, the lattice $t^{-\theta}\Lambda_0=\oplus_nt^{-n}\calO\otimes_{\CC}V_n$. Therefore the action of $B_2$ on $\det(t^{-\theta}\Lambda_0:\Lambda_1)$ is via a character whose pairing with $\theta$ is:
\begin{equation}\label{degLV}
\sum_{n}n(n-[\dfrac{n}{2}])\dim V_n=\sum_n\dfrac{n^2}{2}\dim V_n=\dfrac{1}{2}\sum_{\chi\in\weight(V)}\dim V_{\chi^\vee}\cdot\langle\chi^\vee,\theta\rangle^2.
\end{equation}
Here the first equality follows from the fact that $\dim V_n=\dim V_{-n}$. As a well-known fact for line bundles on $\PP^1=\SL(2)/B_2$, the number in \eqref{degLV} is the degree of restriction of the line bundle $\calL^\phi_{\Lambda_1}$ to the $\SL(2)$-orbit through $t^{-\theta}$, hence the degree of $\calL^\phi_{\Lambda_1}$ on the generating cycle of $H_2(\Grass^0_G)$ (here $\Lambda_1$ can be any reference lattice, because different choices of $\Lambda_1$ give rise to isomorphic line bundles).
\end{proof}


\begin{remark}The lemma shows that some power of $\calL_\det$ admits a natural $G(\calO)$-equivariant structure (the
$G(\calO)$-equivariant structure is necessarily unique if $G$ is almost simple). However, $\calL_\det$ itself may not have a
$G(\calO)$-equivariant structure. Assume that under the isomorphism
$\Pic(\Grass_G^0)\cong\ZZ$, the image of the natural forgetful map
$\Pic^T(\Grass_G^0)\to\Pic(\Grass_G^0)$ is $n\ZZ\subset\ZZ$. Then
the number $n$ coincides with the number $n_G$ which will be introduced in Remark
\ref{nG}.
\end{remark}


\section{The regular element $e^T$}\label{first Chern class}
We continue using the notations introduced in the previous section. Since $\calL^\phi_{\Lambda_0}\cong\calL^{\otimes d_V}_{\det}$ admits a canonical $G$-equivariant structure (hence a canonical $T$-equivariant structure), we can take its equivariant Chern class $c^T_1(\calL^\phi_{\Lambda_0})\in H^2(\Grass_G)$. 
\begin{lemma}
The element $c^T_1(\calL^\phi_{\Lambda_0})$ is primitive in the Hopf algebra $H^*_T(\Grass_G)$.
\end{lemma}
\begin{proof}
Let $m_K:\Omega K\times\Omega K\to\Omega K$ be the multiplication. Then the $T\cap K$-equivariant line bundle $m_K^*(\calL^\phi_{\Lambda_0})$ restricts to $\calL^\phi_{\Lambda_0}$ on both $\Omega K\times\{*\}$ and $\{*\}\times\Omega K$ ($\{*\}$ is the unit element in $\Omega K$). This forces $c^T_1(\calL^\phi_{\Lambda_0})$ to be primitive. Therefore $c^T_1$ is also primitive.
\end{proof}

\begin{cor} The element
\begin{equation*}
c^T_1=\dfrac{1}{d_V}c^T_1(\calL^\phi_{\Lambda_0})\in H^2_T(\Grass_G,\QQ)
\end{equation*}
is independent of the choice of the representation $\phi:G\to\GL(V)$.
\end{cor}
\begin{proof}
For different choices of $(\phi,V)$, the classes $\frac{1}{d_V}c^T_1(\calL^\phi_{\Lambda_0})$ have the same image $c_1(\calL_{\det})$ in $H^2(\Grass_G)$. However, there can be at most one primitive element in $H^2_T(\Grass_G,\QQ)$ with a given image in $H^2(\Grass_G,\QQ)$. 
\end{proof}

\subsection{The element $e^T$} For each object $\calF\in\calP_{\QQ}$, cup product with $c^T_1$ induces a functorial map
\[
\cup c^T_1:H^*_{T}(\calF)\to H^{*+2}_T(\calF).
\]
By Proposition \ref{p:multact}, the fact that $c^T_1$ is primitive implies that for any $\calF_1,\calF_2\in\calP_{\QQ}$, there is a functorial commutative diagram
\[
\xymatrix{H^*_T(\calF_1)\otimes_{R_T}H^*_T(\calF_2)\ar[r]^(0.6){\sim}\ar[d]^{c^T_1\otimes\id+\id\otimes c^T_1} & H^*_T(\calF_1*\calF_2)\ar[d]^{c^T_1}\\
H^*_T(\calF_1)\otimes_{R_T}H^*_T(\calF_2)[2]\ar[r]^(0.6){\sim} &
H^*_T(\calF_1*\calF_2)[2].}
\]

In general, let $c$ is an endomorphism of the functor $\omega_R:\Rep(G^\vee)\to\Mod(R)$ (the composition of the forgetful functor with the tensor product $\otimes R$) satisfying $c_{V_1\otimes V_2}=c_{V_1}\otimes\id_{V_2}+\id_{V_1}\otimes c_{V_2}$, then $c$ determines an element $e\in\frg^\vee\otimes R$. In fact, Consider the fiber functor $\omega_{R,\epsilon}:\Rep(G^\vee)\to\Mod(R[\epsilon]/\epsilon^2)$ sending $V\mapsto V\otimes R[\epsilon]/\epsilon^2=V\otimes R\oplus V\otimes R\epsilon$. The maps
\begin{equation}\label{2by2}
\left( \begin{array}{cc}
\id_V & 0 \\
c_V & \id_V \end{array} \right):V\otimes R\oplus V\otimes R\epsilon\to V\otimes R\oplus V\otimes R\epsilon
\end{equation}
give a tensor automorphism of $\omega_{R,\epsilon}$, hence an element $e\in G^\vee(R[\epsilon]/\epsilon^2)$. Since the induced tensor automorphism of $\omega_R$ (via $R[\epsilon]/\epsilon^2\to R$) is the identity, $e$ lies in the kernel of the reduction map $G^\vee(R[\epsilon]/\epsilon^2)\to G^\vee(R)$, which is $\frg^\vee\otimes R$.

Applying the above discussion to the functor $H^*_T:\calP\to\Mod(R_T\otimes\QQ)$ (which can be identified with the functor $\omega_{R_T\otimes\QQ}:\Rep(G^\vee)\to\Mod(R_T\otimes\QQ)$ by Lemma \ref{l:freeT}) and $c=c^T_1$ , we conclude that $c^T_1$ gives an element $e^T\in\frg^\vee\otimes R_T\otimes\QQ$.

\begin{remark}
More generally, suppose $\omega':\Rep(G^\vee)\to\Mod(R)$ is a fiber functor and $c'$ is an endomorphism of $\omega'$ satisfying $c'_{V_1\otimes V_2}=c'_{V_1}\otimes\id_{\omega'(V_2)}+\id_{\omega'(V_1)}\otimes c'_{V_2}$. Then $\omega'$ determines a $G^\vee$-torsor $\calE$ over $\Spec R$. Let $\delta:\Spec R[\epsilon]/\epsilon^2\to\Spec R$ be the structure map and let $\delta^*\calE$ be the pull back $G^\vee$-torsor to $\Spec R[\epsilon]/\epsilon^2$. Then the same construction as in \eqref{2by2} determines an automorphism of $\delta^*\calE$ whose restriction to the closed subscheme $\Spec R\hookrightarrow \Spec R[\epsilon]/\epsilon^2$ is the identity. Spelling these out, we see that $c'$ determines a Higgs field of the $G^\vee$-torsor $\calE$, i.e., a global section $e'\in\Gamma(\Spec R,\Ad(\calE))$ where $\Ad(\calE)=\calE\twtimes{G^\vee}\frg^\vee$ is the adjoint bundle. In other words, the pair $(\omega',c')$ determines an $R$-point $(\calE,e')$ of the stack $[\frg^\vee/G^\vee]$.

Applying the above discussion to the functor $H^*_G(-)$ (which determines the torsor $\calE$ by Remark \ref{r:torsorE}) and $c^G_1(\calL)$ for any $G$-equivariant line bundle $\calL$ on $\Grass_G$, we get a commutative diagram:
\[
\xymatrix{\Spec R_T\ar[rr]^{c^T_1(\calL)}\ar[d] & & \frg^\vee\ar[d]\\
\Spec R_G\ar[rr]^{(H^*_G(-),c_1^G(\calL))} && [\frg^\vee/G^\vee]}
\]
\end{remark}

In the remaining part of this section, we determine the element $e^T$ explicitly.

We give an (even) grading on $\frg^\vee=\oplus_{n\in\ZZ}\frg^\vee_{2n}$ according to the action of $\GG_m$ via the conjugation of $2\rho^\vee:\GG_m\to T^\vee$:
\[
\frg^\vee_{2n}=\bigoplus_{\alpha\in\Phi,\langle2\rho^\vee,\alpha\rangle=2n}\frg^\vee_{\alpha}.
\]
We give $\frg^\vee\otimes R_T$ the tensor product (even) grading. The fact that $\cup c^T_1$ increases the cohomological degree by $2$ implies that $e^T$ is homogeneous of degree $2$ in $\frg^\vee\otimes R_T$, i.e.,
\[
e^T\in\bigoplus_{n\in\ZZ}\frg^\vee_{2(1-n)}\otimes\Sym_n(\xch(T))\otimes\QQ.
\]

\begin{lemma}
\begin{equation}\label{eT}
e^T\in\frg^\vee_2\otimes\QQ\oplus\frt^\vee\otimes\xch(T)\otimes\QQ.
\end{equation}
\end{lemma}
\begin{proof}
The action of $\cup c_1^T$ preserves the $T$-equivariant MV-filtration on $H^*_T(\calF)$ which gives the canonical $B^\vee$-reduction of the trivial $G^\vee$-torsor $\calE|_{\Spec R_T}$ (see Remark \ref{r:torsorE}). Hence, $e^T$ is an $R_T\otimes\QQ$-valued element of the Lie algebra $\frb^\vee$. Since $e^T$ is homogeneous of degree 2 in $\frb^\vee\otimes R_T\otimes\QQ$, it must take the form \eqref{eT}.
\end{proof}

Write $e^T=e+f$ where $e\in\frg^\vee_2\otimes\QQ$ and $f\in\frt^\vee\otimes\xch(T)\otimes\QQ$.

We first calculate $e$. Clearly, $e$ is obtained from the endomorphism of the functor $H^*:\calP\to\Mod(\ZZ)$ given by $\cup c_1(\calL_{\det})$, the ordinary Chern class of $\calL_{\det}$. In particular, $e\in\frg^\vee$. For each simple root $\alpha_i$ of $\frg^\vee$, let $x_i$ be the element of a Chevalley basis of $\frg^\vee$ in the root space $\frg^\vee_{\alpha_i}$. Then $e$ is an integral combination of the $x_i$'s.

\begin{prop}\label{p:e0} Up to changing the Chevalley basis elements $\{x_i\}$ by signs, we have
\begin{equation}\label{e0}
e=\sum_{i}|\alpha_i|^2x_i
\end{equation}
where $|\ |^2$ is a $W$-invariant quadratic form on $\xcoch(T)$ normalized so that short coroots (of $G$) have length one (hence $|\alpha_i|^2=1,2$ or $3$).
\end{prop}
\begin{proof}
Let $\calI^{\lambda}_*$ be the perverse sheaf $\leftexp{p}{j}^\lambda_*\ZZ[\langle2\rho^\vee,\lambda\rangle]$ where $j^\lambda:\Grass_\lambda\to\Grass_G$ is the inclusion. We first claim that the restriction map $H^{i-\langle2\rho^\vee,\lambda\rangle}(\calI^{\lambda}_*)\to H^i(\Grass_\lambda)$ is an isomorphism for $i\leq2$. In fact, the cone $\calF$ of $\calI^\lambda_*\to j^{\lambda}_*\ZZ[\langle2\rho^\vee,\lambda\rangle]$ has perverse degree $\geq1$, and is supported on $\Grass_{<\lambda}=\Grass_{\leq\lambda}-\Grass_\lambda$, which has dimension $\leq\langle2\rho^\vee,\lambda\rangle-2$. Therefore, $H^i(\calF)$ vanishes for $i<-\dim{\Grass_{<\lambda}}+1\leq-\langle2\rho^\vee,\lambda\rangle+3$. This implies our assertion.

Now we prove the proposition. Without changing $e$ and $\frg^\vee$, we can assume that $G$ is of
adjoint type. Write $e=\sum_{i}n_ix_i$ for some integers $n_i$. For
each simple coroot $\alpha_i$ of $G$, let $\lambda_i$ be the
corresponding fundamental coweight. By \cite[Proposition 13.1]{MV},
the $G^\vee$-module $H^*(\calI^{-w_0(\lambda_i)}_*)$ is the Schur
module with highest weight $-w_0(\lambda_i)$, where $w_0$ stands for
the element of longest length in the Weyl group of $G$. Therefore,
the Lie algebra element $x_i$ sends the lowest weight vector $v_{low}$ (with
weight $-\lambda_i$) to a generator of the rank 1 weight space
of weight $-\lambda_i+\alpha_i$. For $j\neq i$, clearly $x_j\cdot v_{low}=0$. Therefore $e\cdot v_{low}$ is $n_i$-times a generator of the rank 1 weight space of weight $-\lambda_i+\alpha_i$.

Translating to geometry, this means that the action of
$c_1(\calL_{\det})$ sends the fundamental class
$[\Grass_{-\lambda_i}]\in
H^{-\langle2\rho^\vee,-w_0(\lambda_i)\rangle}(\calI^{-w_0(\lambda_i)}_*)$
to $n_i$-times a generator of $H^{2-\langle
2\rho^\vee,-w_0(\lambda_i)\rangle}(\calI^{-w_0(\lambda_i)}_*)$
(which is a free $\ZZ$-module of rank one). By the discussion at the
beginning of the proof, we can identify $H^{i-\langle2\rho^\vee,-w_0(\lambda_i)\rangle}(\calI^{-w_0(\lambda_i)}_*)$ with $H^i(\Grass_{-\lambda_i})$. Hence $c_1(\calL_{\det})$ sends $1\in
H^0(\Grass_{-\lambda_i})$ to $n_i$-times a generator of
$H^2(\Grass_{-\lambda_i})$ (which is unique up to sign). To
calculate this number $n_i$, it suffices to calculate the class
$c_1(\calL_{\det})|_{\Grass_{-\lambda_i}}\in
H^2(\Grass_{-\lambda_i})$, or its further restriction to $H^2(G\cdot
t^{-\lambda_i})$ (because the inclusion $G\cdot
t^{-\lambda_i}\hookrightarrow\Grass_{-\lambda_i}$ is a homotopy
equivalence).

Now the restriction of $c_1(\calL_{\det})$ to $G\cdot
t^{-\lambda_i}$ is $n_i$-times the generator of $H^2(G\cdot
t^{-\lambda_i})$, which is a free $\ZZ$-module of rank 1. Let
$\phi_i:\SL(2)\to G$ be the homomorphism corresponding to the coroot
$\alpha_i$ of $G$. Then $\phi_i(\SL(2))\cdot
t^{-\lambda_i}=\SL(2)/B_2\cong\PP^1$ is the generating 2-cycle of the partial flag variety
$G\cdot t^{-\lambda_i}$ (by the Bruhat decomposition), and $n_i$
equals the degree of the restriction of $\calL_{\det}$ to
$\phi_i(\SL(2))\cdot t^{-\lambda_i}$.

Let $\Ad:G\to GL(\frg)$ denote the adjoint representation of $G$. As a fact for line bundles on $\PP^1$, the degree of $\calL^{\Ad}_{\Lambda_0}$ on $\phi_i(\SL(2))\cdot t^{-\lambda_i}\cong\PP^1$ is the same as the pairing of $c_1^T(\calL^{\Ad}_{\Lambda_0})|_{t^{-\lambda_i}}\in H^2_T(\{t^{-\lambda_i}\})=\xch(T)$ with $\alpha^\vee_i$. By \eqref{reschern}, this degree equals
\[
-\sum_{\alpha^\vee\in\Phi^\vee}\langle\alpha^\vee,-\lambda_i\rangle\langle\alpha^\vee,\alpha_i\rangle=(\lambda_i,\alpha_i)_{Kil}.
\]
By Lemma \ref{l:deg}, $c_1(\calL^{\Ad}_{\Lambda_0})=d_{\Ad}c_1(\calL_{\det})=\dfrac{1}{2}(\theta,\theta)_{Kil}c_1(\calL_{\det})$, therefore the degree of $\calL_{\det}$ on $\phi_i(\SL(2))\cdot t^{-\lambda_i}$ is
\[
n_i=\dfrac{2(\lambda_i,\alpha_i)_{Kil}}{(\theta,\theta)_{Kil}}=\dfrac{(\alpha_i,\alpha_i)_{Kil}}{(\theta,\theta)_{Kil}}.
\]
Since $\theta$ is a short coroot, we conclude that $n_i=|\alpha_i|^2$.
\end{proof}

Next we determine the element $f\in\frt^\vee\otimes\xch(T)\otimes\QQ$.

\begin{prop}\label{p:f}
Identifying $\frt^\vee$ with $\xch(T)$ and $\frt^\vee\otimes\xch(T)$ with the dual of $\xcoch(T)\otimes\xcoch(T)$, the element $f$ is the ($\QQ$-valued) symmetric bilinear form
\[
f(\lambda,\mu)=\dfrac{-2(\lambda,\mu)_{Kil}}{(\theta,\theta)_{Kil}}.
\]
\end{prop}
\begin{proof}Write $f=\sum f_i\otimes g_i$, where $f_i\in\frt^\vee, g_i\in\xch(T)$. Since the cup product with $c^T_1$ preserves the $T$-equivariant MV-filtration, there is an induced action $\overline{e^T}$ on the associated graded pieces $H^*_{T,c}(S_\lambda,-)$. The element $f$ is given by the action $\overline{e^T}$. More precisely, for $\lambda\in\xcoch(T)$, the action of $\overline{e^T}$ on the $\lambda$-weight subspace is given by the multiplication of $\sum\langle f_i,\lambda\rangle g_i$. Translating back to geometry, this means that the cup product of $c^T_1$ on $H^*_{T,c}(S_\lambda,-)$ should be the multiplication by $\sum\langle f_i,\lambda\rangle g_i$ (where as before, $g_i\in\xch(T)\cong H^2_T(\pt)$). On the other hand, $c^T_1$ acts on $H^*_{T,c}(S_\lambda,-)$ via its restriction to  $S_\lambda$. Since $\{t^\lambda\}\hookrightarrow S_\lambda$ is a $T$-equivariant homotopy equivalence, the action of $c^T_1$ on $H^*_{T,c}(S_\lambda,-)$ is given by the multiplication of $i^*_\lambda(c^T_1)\in H^2_T(\{t^\lambda\},\QQ)=\xcoch(T)\otimes\QQ$.

Applying \eqref{reschern} to the line bundle $\calL^{\Ad}_{\Lambda_0}$ associated with the adjoint representation $\Ad$ of $G$, and using Lemma \ref{l:deg}, we see that the restriction of $c^T_1$ to $H^2_T(\{t^\lambda\})$ is given by:
\[
i^*_{\lambda}(c^T_1)=\dfrac{-2}{(\theta,\theta)_{Kil}}\sum_{\alpha^\vee\in\Phi^\vee}\langle\alpha^\vee,\lambda\rangle\alpha^\vee.
\]
Therefore,
\[
f=\dfrac{-2}{(\theta,\theta)_{Kil}}\sum_{\alpha^\vee\in\Phi^\vee}\alpha^\vee\otimes\alpha^\vee.
\]
If we view $f$ as a bilinear form on $\xcoch(T)$, then for any $\lambda,\mu\in\xcoch(T)$,
\[
f(\lambda,\mu)=\dfrac{-2\sum_{\alpha^\vee\in\Phi^\vee}\langle\alpha^\vee,\lambda\rangle\langle\alpha^\vee,\mu\rangle}{(\theta,\theta)_{Kil}}=\dfrac{-2(\lambda,\mu)_{Kil}}{(\theta,\theta)_{Kil}}.
\]
\end{proof}

\begin{remark}\label{nG}Let $n_G$ be the least positive integer such that
$n_Gf\in\frt^\vee\otimes\xch(T)$ (or equivalently, such that
$n_Ge^T\in\frg^\vee\otimes R_T$). Clearly, we have
$n_G=n_{G^{der}}$. On
the other hand, for $G$ is almost simple, this number can be found in
\cite[Table C]{Sor}, in which $n_G$ is denoted by $\ell_b$. We reproduce these numbers explicitly in the following list.

\begin{itemize}
\item $n_G=\dfrac{\#\pi_1(G)}{\gcd(\#Z(G),\#\pi_1(G))}$ if $G$ is of type A;
\item $n_G=2$ if $G$ is $\mathbf{PSp}_{2m}$ (for $m$ odd), $\mathbf{SO}_{4m}^{\pm}$ (for $m$ odd), $\mathbf{PSO}_{4m}$, or $\mathbf{PE}_7$;
\item $n_G=3$ if $G$ is  $\mathbf{PE}_6$;
\item $n_G=4$ if $G$ is $\mathbf{PSO}_{4m+2}$;
\item $n_G=1$ otherwise.
\end{itemize}
Here $\mathbf{P}$ stands for the adjoint form, and $\mathbf{SO}^{\pm}_{4m}$ are the two quotients of $\mathbf{Spin}_{4m}$ by $\ZZ/2$ which are not isomorphic to $\mathbf{SO}_{4m}$.
\end{remark}


\section{Proof of the main result}\label{proof}
Recall the morphism $\tsigma^T:\Spec H^T_*(\Grass_G)\to B^\vee\times\Spec R_T$ defined in Proposition \ref{p:ZU}(1) and its non-equivariant counterpart $\tsigma:\Spec H_*(\Grass_G)\to B^\vee$. It is clear that the action of $\sigma_{can}$ on the fiber functor $H^*_T(-)\otimes_{R_T}H^T_*(\Grass_G)$ commutes with the cup product action of $c^T_1$. Therefore, $\tsigma^T$ further factors through the homomorphism
\begin{equation*}
\sigma^T:\Spec H^T_*(\Grass_G)[1/n_G]\to B^\vee_{e^T},
\end{equation*}
of group schemes over $\Spec R_T[1/n_G]$. Here $B^\vee_{e^T}$ is the centralizer group scheme of $e^T\in\frg^\vee(R_T[1/n_G])$ in $B^\vee\times\Spec R_T[1/n_G]$. Since $e$ is already an element in $\frg^\vee$, we also have the homomorphism
\begin{equation*}
\sigma:\Spec H_*(\Grass_G)\to Z^\vee\times U^\vee_e\hookrightarrow B^\vee_e
\end{equation*}
of group schemes over $\ZZ$ (the first arrow above follows from Proposition \ref{p:ZU}(4)).

\begin{theorem}\label{th:main} Assume that $G^{der}$ is almost simple.
\begin{enumerate}
\item The homomorphism $\sigma:\Spec H_*(\Grass_G)\to B^\vee_e$ is a closed embedding, and is an isomorphism over $\Spec\ZZ[1/\ell_G]$, where $\ell_G=1,2$ or $3$ is the square of the ratio of lengths of the long and short coroots of $G$. In particular, $\Spec H_*(\Grass_G)\to Z^\vee\times U^\vee_e$ is also an isomorphism over $\Spec\ZZ[1/\ell_G]$.
\item The homomorphism $\sigma^T:\Spec H^T_*(\Grass_G)[1/n_G]\to B^\vee_{e^T}$ is a closed embedding, and is an isomorphism over $\Spec R_T[\dfrac{1}{\ell_Gn_G}]$.
\end{enumerate}
\end{theorem}
\begin{proof}
Step I. We first prove that $\sigma^T$ and $\sigma$ are closed embeddings, i.e., $\tau^T:\calO(B^\vee_{e^T})\to H^T_*(\Grass_G)[1/n_G]$ and $\tau:\calO(B^\vee_e)\to H_*(\Grass_G)$ are surjective. It suffices to prove the surjectivity of $\tau$ because any homogeneous lifting of a $\ZZ[1/n_G]$-basis of $H_*(\Grass_G)[1/n_G]$ to $H^T_*(\Grass_G)[1/n_G]$ generates $H^T_*(\Grass_G)[1/n_G]$ as an $R_T[1/n_G]$-module, and we can choose these liftings to be in the image of $\tau^T$ since $\calO(B^\vee_{e^T})\to\calO(B^\vee_e)$ is surjective.

Fix a dominant coweight $\lambda$ of $G$. Consider the following diagram
\begin{equation}\label{trisq}
\xymatrix{H_*(\Grass_\lambda)\ar[r]\ar[dr] & H_{*-\langle2\rho^\vee,\lambda\rangle}(\calI^\lambda_*)\ar[d]\ar[r]^(0.6){\beta} & \calO(B^\vee_{e})\ar[d]^{\tau}\\
&  H_*(\Grass_{\leq\lambda})\ar[r] & H_*(\Grass_G)}
\end{equation}
where $H_*(\calI^\lambda_*)$ is the $\ZZ$-dual of $H^*(\calI^\lambda_*)$. In the following, we describe the various arrows in this diagram and verify its commutativity.

First we concentrate on the triangle in \eqref{trisq}. We have the natural morphisms between complexes on $\Grass_{\leq\lambda}$:
\[
\ZZ\to\calI^\lambda_*[-\langle2\rho^\vee,\lambda\rangle]\to j^\lambda_*\ZZ.
\]
Taking duals of their cohomology, we get the triangle in \eqref{trisq}, which is therefore commmutative.

Next we examine the square in \eqref{trisq}. The map $\beta$ is defined to be the matrix coefficient of the $G^\vee$-module $H^*(\calI^\lambda_*)$ with respect to the lowest weight vector
\[
v_{low}=[\Grass_{\leq\lambda}]\in H^{-\langle2\rho^\vee,\lambda\rangle}(\calI^\lambda_*).
\]
More precisely, for $u\in H_*(\calI^\lambda_*)$,
\[
\beta(u)(g)=\langle u,g\cdot v_{low}\rangle, \forall g\in B^\vee_{e}.
\]
We check that commutativity of the square in \eqref{trisq}. For $u\in H_*(\calI^\lambda_*)$ and the $H_*(\Grass_G)$-valued point $\sigma_{can}\in B^\vee_{e}(H_*(\Grass_G))$, we have
\begin{equation}\label{taubeta}
\tau\beta(u)=\beta(u)(\sigma_{can})=\langle u,\sigma_{can} v_{low}\rangle=\sum_i\langle u,h^i\cdot v_{low}\rangle h_i\in H_*(\Grass_G)
\end{equation}
where $\{h^i\},\{h_i\}$ are dual bases of $H^*(\Grass_G)$ and $H_*(\Grass_G)$. On the other hand, the map $H^*(\Grass_{\leq\lambda})\to H^{*-\langle2\rho^\vee,\lambda\rangle}(\calI^\lambda_*)$ sends $h\mapsto h\cdot v_{low}$, hence the adjoint map $H_{*-\langle2\rho^\vee,\lambda\rangle}(\calI^\lambda_*)\to H_*(\Grass_{\leq\lambda})$ sends $u\mapsto\sum_i\langle u,h^i\cdot v_{low}\rangle h_i$. Combining with \eqref{taubeta}, we have verified that the square in \eqref{trisq} is commutative.

From diagram \eqref{trisq}, we conclude that the image of $\calO(B^\vee_e)\to H_*(\Grass_G)$ contains the image of $H_*(G\cdot t^\lambda)=H_*(\Grass_\lambda)\to H_*(\Grass_G)$ for any $\lambda\in\xcoch(T)$. When $G$ is of adjoint type, by \cite[Theorem 1]{Bott}, there exists a single $\lambda$ (called the ``generating circle'' by Bott), such that the image of $H_*(G\cdot t^\lambda)\to H_*(\Grass_G)$ generates $H_*(\Grass_G)$ as a $\ZZ$-algebra. This implies that $\calO(B^\vee_e)\to H_*(\Grass_G)$ is surjective in the case $G$ is adjoint. By Proposition \ref{p:ZU}(4), this means that $\Spec H_*(\Grass_G)=\Spec H_0(\Grass_G)\otimes H_*(\Grass_{G^{sc}})\to Z^\vee\times U^\vee$ is a closed embedding when $G$ is adjoint. Since $\Spec H_0(\Grass_G)\cong Z^\vee$, we conclude that $H_*(\Grass_{G^{sc}})\to U^\vee$ is also a closed embedding. For general $G$, using Proposition \ref{p:ZU}(4) again, we conclude that $\Spec H_*(\Grass_G)\to Z^\vee\times U^\vee$ is always a closed embedding.

Step II. We prove that $B^\vee_{e^T}[1/\ell_G]$ is flat over $R_T':=R_T[\dfrac{1}{\ell_Gn_G}]$. Consider the morphism over $R'_T$
\[
\phi=\Ad(-)e^T-e^T: B^\vee\times\Spec R'_T\to\fru^\vee\times\Spec R'_T.
\]
Then $B^\vee_{e^T}[1/\ell_G]=\phi^{-1}(0)$, i.e.,
$B^\vee_{e^T}[1/\ell_G]$ is the closed subscheme of the
$B^\vee\times\Spec R_T'$ (which is smooth over $R_T'$) cut out by
$\dim\fru^\vee$ equations. Hence the fiber dimensions of
$B^\vee_{e^T}\to\Spec R_T'$ are at least $\dim B^\vee-\dim\fru^\vee=r$,
the rank of $G^\vee$. If we can show that all the fibers of
$B^\vee_{e^T}[1/\ell_G]\to\Spec R_T'$ are $r$-dimensional, then
$B^\vee_{e^T}[1/\ell_G]$ will be locally a complete intersection
over $R_T'$. Using that all the fibers are of the same dimension
again and that $R'_T$ is regular, we can conclude that
$B^\vee_{e^T}$ is flat over $R_T'$ (see \cite[20.D]{M} with obvious
modifications). Therefore, it suffices to show that for any closed
point $s\in\Spec R'_T$, the fiber $(B^\vee_e)_s$ has dimension $r$.

Since $R_T$ is graded, $\Spec R_T$ admits a natural $\GG_m$-action with fixed point locus $z:\Spec\ZZ\hookrightarrow\Spec R_T$ defined by the augmentation ideal of $R_T$. Moreover, this fixed point locus is attracting: for each point $s\in\Spec R_T$, let $\eta\in\Spec R_T$ be the generic point of the $\GG_m$-orbit of $s$, then the closure of $\eta$ intersects the fixed point locus $z(\Spec\ZZ)$. In other words, $\eta$ specializes to a point $s_0\in z(\Spec\ZZ)$ of the same residue characteristic as $s$. Hence $\dim(B^\vee_e)_s=\dim(B^\vee_e)_{\eta}\leq\dim(B^\vee_{e_0})_{s_0}$. Therefore, in order to show that all fibers of $B^\vee_{e^T}[1/\ell_G]\to\Spec R'_T$ have dimension $r$, it suffices to show that all (geometric) fibers of $B^\vee_{e}[1/\ell_G]\to\Spec\ZZ[1/\ell_G]$ have dimesion $r$.

Let $k$ be an algebraic closure of $\FF_p$ for some prime $p$ not
dividing $\ell_G$. We would like to show that
$B^\vee_{e}\otimes_{\ZZ}k$ has dimension $r$ over $k$. We base
change the situation from $\ZZ$ to $k$ without changing notations.
So in the rest of this paragraph $B^\vee_{e}$ is over $k$, etc. By
Proposition \ref{p:e0}, the element $e$ has the form
$\sum_{i}n_ix_i$ where each $n_i\neq0$ in $k$. We can choose $t\in
T^\vee(k)$ such that $\Ad(t)e=\sum_ix_i$ (because
$\prod\alpha_i:T^\vee(k)\to (k^\times)^r$ is surjective). Therefore
we only need to treat the element $e_1=\sum_ix_i$. It is well-known
that $B^\vee_{e_1}=Z^\vee\times U^\vee_{e_1}$, therefore it suffices
to show that $\dim U^\vee_{e_1}=r-\dim Z^\vee$. This equality
follows from the main result of \cite{K}. This proves that
$B^\vee_{e^T}$ is flat over $R_T'$.

The above argument also shows that $B^\vee_e[1/\ell_G]$ is flat over $\ZZ[1/\ell_G]$.

Step III. Now we can finish the proof.

(2) The equivariant version. In view of Step I, it remains to prove that $\calO(B^\vee_{e^T})[1/\ell_G]\to H^T_*(\Grass_G)[\dfrac{1}{\ell_Gn_G}]$ is injective. Since both $\calO(B^\vee_{e^T})[1/\ell_G]$ and $H^T_*(\Grass_G)[\dfrac{1}{\ell_Gn_G}]$ are flat over $R'_T$, it suffices to show that $\calO(B^\vee_e)[1/\ell_G]\otimes_{R'_T}Q\to H^T_*(\Grass_G)\otimes_{R_T}Q$ is injective where $Q=\Frac(R_T)$.

Recall from Remark \ref{r:loc} that we have a commutative diagram
\[
\xymatrix{\Spec(H^T_*(\Grass_G)\otimes_{R_T}Q)\ar[r]^(.7){\tau_Q}\ar[dr]^{\Spec(Loc_*)} & B^\vee_{e_Q}\ar[d]\\ & T^\vee\times\Spec Q}
\]
where $e_Q\in\frb^\vee(Q)$ is the value of $e^T$ over the generic point of $\Spec R_T$. By the equivariant localization theorem, the localization map $Loc_*$ in \eqref{homloc} is an isomorphism. On the other hand, by Proposition \ref{p:f}, $e_Q$ is a regular semisimple element in $\frb^\vee(Q)$, hence the projection $B^\vee_{e_Q}\to T^\vee\times\Spec Q$ is also an isomorphism. This shows that the closed embedding $\tau_Q$ is in fact an isomorphism, and (2) is proved.

(1) The non-equivariant version. As a consequence of the equivariant version, $\sigma:\Spec H_*(\Grass_G)\to B^\vee_e$ is an isomorphism over $\QQ$ (we base change $\sigma^T$ to $\Spec\QQ\to\Spec\ZZ\xrightarrow{z}\Spec R_T$). Therefore, $\tau\otimes\QQ:\calO(B^\vee_e)\otimes\QQ\to H_*(\Grass_G,\QQ)$ is an isomorphism. Since $\calO(B^\vee_e)[1/\ell_G]$ is flat over $\ZZ[1/\ell_G]$ by Step II, hence torsion-free, we conclude that $\tau[1/\ell_G]:\calO(B^\vee_e)[1/\ell_G]\to H_*(\Grass_G)[1/\ell_G]$ is injective. Combining this with Step I, (1) is proved.

\end{proof}

\begin{remark}
\begin{enumerate}
\item []
\item Inverting $\ell_G$ in the statement of the above theorem is necessary because $e$ is not regular modulo $\ell_G$, hence the fiber of $B^\vee_e$ over $\FF_{\ell_G}$ does not have the same dimension as $\Spec H_*(\Grass_G,\FF_{\ell_G})$. However, the proof of the above theorem implies that $\Spec H_*(\Grass_G)$ is always isomorphic to the Zariski closure in $B_e^\vee$ of its generic fiber.
\item The theorem shows that no primes other than $\ell_G$ need to be inverted, even for those ``bad primes'' of $G$ (i.e., those dividing the coefficients of the highest root in terms of a linear combination of simple roots). For example, let $G$ be the exceptional group $G_2$. Then Bott (cf. \cite{Bott}) shows that
\[H_*(\Omega K)\cong \ZZ[u,v,w]/(2v-u^2).\]
On the other hand, one can show (using formulae in \cite{K}) that
\[\calO(B^\vee_e)\cong\ZZ[A,C,D,E,F]/(2C-3A^2,3(D-3A^3),3(E+C^2-9A^4))\]
We find that $H_*(\Omega K,\FF_2)\cong\calO(B^\vee_e)\otimes\FF_2$ although 2 is considered as a bad prime for group $G_2$.
\end{enumerate}
\end{remark}

\subsection{Description of the cohomology}\label{dist} Dually, we have a description of the cohomology of $\Grass_G$. Recall that if $H$ is an affine group scheme over some commutative ring $k$, the algebra of distributions on $H$ (cf. \cite[\S I.7]{J}) is defined as
\[
\Dist(H)=\{f:\calO(H)\to k | f \mbox{ is }k\mbox{-linear}, f(\frm^{n+1})=0 \mbox{ for some } n\}.
\]
Here $\frm$ is the augmented ideal of $\calO(H)$. It has a natural algebra structure dual to the coalgebra structure on $\calO(H)$. $\Dist(H)$ can be also regarded as the algebra of left invariant differential operators on $H$. If, in addition, $H$ is infinitesimally flat (i.e. $\calO(H)/\frm^{n+1}$ is flat over $k$ for any $n$), then $\Dist(H)$ has a natural structure of a Hopf algebra over $k$.

Now let $H=B_e^\vee$. It is proved by Springer (cf. \cite{Spr}) that
$B^\vee_e\otimes\FF_p$ is smooth if the prime $p$ is good (i.e. not bad) for
$G^\vee$. This fact, together with the flatness of $B^\vee$ over
$\ZZ[1/\ell_G]$ as shown in the course of proving Theorem \ref{th:main},
implies that $B^\vee_e$ is smooth over $\ZZ'$, where $\ZZ'$ is
obtained from $\ZZ$ by inverting the bad primes of $G^\vee$. In
particular, $B^\vee_e$ is infinitesimally flat over $\ZZ'$.

\begin{cor}
Assume that $G$ is almost simple and simply-connected. Then there is an isomorphism of algebras over $\ZZ$
\begin{equation}\label{cohotau}
\tau^*:H^*(\Grass_G)\isom\Dist(U^\vee_{e}),
\end{equation}
which is an isomorphism of Hopf algebras over $\ZZ'$.
\end{cor}
\begin{proof}
Observe that $Z^\vee$ is trivial in this case, hence $U^\vee_e=B^\vee_e$. In the course of proving Theorem \ref{th:main} we have shown that $\tau:\calO(U^\vee)\to H_*(\Grass_G)$ is an isomorphism up to torsion. Observe that $\Dist(U^\vee_e)$ can be identified with the graded dual of $\calO(U^\vee_e)$ because elements in the augmentation ideal of $\calO(U^\vee_e)$ all have positive degrees. Taking the graded dual of $\tau$ we get the desired isomorphism $\tau^*$ in \eqref{cohotau}.
\end{proof}

\subsection{Description of the $G$-equivariant homology}\label{ss:Geq} In this subsection, we assume that $G$ is almost simple and simply-connected. We now interpret the $G$-equivariant homology $H^G_*(\Grass_G)$ in terms of the regular centralizer group scheme of $G^\vee$. Let us recall that a prime $p$ is called a torsion prime of $G$ if $H^*(G,\ZZ)$ has $p$-torsion. Let $S_1$ be the multiplicative set generated by the torsion primes of $G$, and let $\ZZ_{S_1}$ be the corresponding localization, i.e., the localization away from the torsion primes. Borel (see \cite{Borel1}) proved that
\begin{equation}\label{RGRT}
R_G\otimes\ZZ_{S_1}\isom (R_T\otimes\ZZ_{S_1})^W. 
\end{equation}
Since $N_G(T)$ acts on $\Grass_G$, it induces an action of $W=N_G(T)/T$ on $H^T_*(\Grass_G)$ compatible with the natural $W$-action on $R_T$. The natural map $H^G_*(\Grass_G)\to H^T_*(\Grass_G)$ then factors through the $W$-invariants of $H^T_*(\Grass_G)$. Using the isomorphism \eqref{RGRT}, it is easy to show that
\begin{equation*}
H^G_*(\Grass_G,\ZZ_{S_1})\isom H^T_*(\Grass_G,\ZZ_{S_1})^W
\end{equation*}
as Hopf algebras over $R_G\otimes\ZZ_{S_1}$.

On the other hand, let $S_2$ be the multiplicative set generated by bad primes of $G^\vee$ and those  dividing $n+1$ if $G$ is of type $A_n$. Clearly $\ell_G,n_G\in S_2$, and the result in \cite{Borel2} implies $S_1\subset S_2$. Let $\ZZ_{S_2}$ be localization of $\ZZ$ with respect to $S_2$. Consider the universal centralizer scheme $I^\vee$ over $\frg^\vee$ (i.e., the fiber of $I^\vee$ over $x\in\frg^\vee$ is the centralizer $G^\vee_x$). After inverting primes in $S_2$, we have the following identifications
\begin{equation}\label{Winv}
(\frg^\vee/\!\!/G)\times\Spec\ZZ_{S_2}\isom(\frt^\vee/\!\!/W)\times\Spec\ZZ_{S_2}\isom(\Spec R_G\otimes\ZZ_{S_2})
\end{equation}
where the latter isomorphism uses the element $f$ in Proposition \ref{p:f}. It can be shown that the restriction of $I^\vee$ to the regular locus $(\frg^\vee)^{reg}\times\Spec\ZZ_{S_2}$ descends to a smooth group scheme $J^\vee$ over $\Spec(R_G\otimes\ZZ_{S_2})\cong(\frg^\vee/\!\!/G)\times\Spec\ZZ_{S_2}$. In fact, this is essentially proved in \cite{Ngo} with all primes dividing $\# W$ inverted. But further argument shows that only primes in $S_2$ need to be inverted.

\begin{prop}
There is a natural isomorphism of group schemes over $\Spec(R_G\otimes\ZZ_{S_2})$:
\begin{equation*}
\Spec H^G_*(\Grass_G,\ZZ_{S_2})\isom J^\vee.
\end{equation*}
\end{prop}
\begin{proof}[Sketch of proof]
Let $\pi:\Spec(R_T\otimes\ZZ_{S_2})\to\Spec(R_G\otimes\ZZ_{S_2})$, then by the definition of the regular centralizer, we have a closed embedding $\beta:B^\vee_{e^T}\to G^\vee_{e^T}=\pi^*J^\vee$ (from now on we view $B^\vee_{e^T}$ as a scheme over $\Spec(R_T\otimes\ZZ_{S_2}$)). Using the flatness of $B^\vee_{e^T}$ and $\pi^*J^\vee$ over $R_G\otimes\ZZ_{S_2}$, it is easy to see that $\beta$ is an isomorphism. Therefore the coordinate ring $\calO(J^\vee)$ can be identified with the $W$-invariants of $\calO(B^\vee_{e^T})\isom\calO(G^\vee_{e^T})$. Recall from \eqref{Winv} that $H^G_*(\Grass_G,\ZZ_{S_2})$ is also identified with the $W$-invariants of $H^T_*(\Grass_G,\ZZ_{S_2})$. It remains to show that the $W$-actions on $H^T_*(\Grass_G,\ZZ_{S_2})$ and $\calO(B^\vee_{e^T})$ correspond to each other under the isomorphism $\sigma^T$ in Theorem \ref{th:main}(2). For this, it is enough to argue over $\Frac(R_T)$ and use equivariant localization (Remark \ref{r:loc}) again. Details are left to the reader.
\end{proof}

\end{document}